\documentclass[12pt,a4paper,reqno]{amsart}
\usepackage{amsthm}
\usepackage{amsmath}
\usepackage{color}
\usepackage{graphicx}
\usepackage{amsfonts}
\usepackage{xfrac}
\usepackage{mathrsfs}
\numberwithin{equation}{section}
\begin{document}
\title{From BCZ map to a discretized analog of the RH}
\author{Y.~Li}
\begin{abstract}
We investigate the properties of the BCZ map. Based on our findings, we define the moduli space associated with its excursions. Subsequently, we utilize the framework we build to establish a discretized analog of the Riemann hypothesis (RH) that holds in a stronger sense from a dynamical perspective. The analog is founded upon a reformulation of the RH, specifically in terms of estimates of $L^1$-averages of BCZ cocycle along periodic orbits of the BCZ map. The primary tool we will rely on is the generalized arithmetic sequence, which we will define and discuss.
\end{abstract}
\maketitle

\newtheorem{ddd}{Definition}[section]
\newtheorem{ttt}[ddd]{Theorem}
\newtheorem{rrr}[ddd]{Remark}
\newtheorem{ppp}[ddd]{Proposition}
\newtheorem{ccc}[ddd]{Corollary}
\newtheorem{llll}[ddd]{Lemma}
\newtheorem{con}[ddd]{Conjecture}
\renewcommand{\bold}[1]{\smallskip \noindent {\bf \boldmath #1 }\nopagebreak[4]}

\section{Introduction}\label{sec:intro}
\raggedbottom
\allowdisplaybreaks

In 1924, J\'er\^ome Franel \cite{Fr} and Edmund Landau \cite{La} showed that the Riemann hypothesis (RH) is equivalent to the statement that for all $\epsilon>0$,
\begin{equation}
\sum\limits_{i=1}^{A_n}\left|\rho_i-\frac{i}{A_n}\right|=O\left(n^{\frac{1}{2}+\epsilon}\right),   \label{40}
\end{equation}
where $\mathcal{F}(n):(\rho_i)_{i=0}^{A_n}$ is the Farey sequence of order $n$. In the sequel, $\epsilon$ will always be used in this context, and the proviso ``for all $\epsilon>0$" will not be repeated.\\

\bold{BCZ map.}
(\ref{40}) showed the strong connection between the RH and the Farey sequence and suggested the possibility of studying gaps statistics of the Farey sequence as an approach to the RH, which attracted many mathematicians to study the Farey sequence's properties. In \cite{Ha70}, R.R.Hall studied the asymptotic analysis of the sum of the square of the 1-gap difference of the Farey sequence and computed the distribution of the gaps of the Farey sequence which was later referred to as Hall's distribution in \cite{At}. When proving a conjecture of R.R.Hall \cite{Ha94} about the asymptotic analysis of the sum of the square of the h-gap difference of the Farey sequence, Boca, Cobeli, and Zaharescu \cite{Bo} constructed the BCZ map:
$$T(a,b)=\left(b,\left[\frac{1+a}{b}\right]b-a\right),$$
which is defined on the Farey triangle:
$$\Omega:=\left\{(a,b)\in\mathbb{R}^2 \mid a,b\in(0,1],a+b>1\right\}.$$
We call $k(a,b):=\left[\frac{1+a}{b}\right]$ the itinerary function (or index function) of $(a,b)$.

There are multiple results concerning various statistical properties of the Farey sequence following the construction of the BCZ map. Augustin-Boca-Cobeli-Zaharescu \cite{Au} studied the h-spacing distribution between Farey points. Boca-Cobeli-Zaharescu \cite{Bo1} showed the distribution of the Farey sequence with odd denominators. Hall-Shiu \cite{sh}, Hall \cite{Ha1} gave some basic properties about the index of the Farey sequence, while Boca-Gologan-Zaharescu \cite{Go} demonstrated some asymptotic formulas regarding the distribution of the index function of the Farey sequence.

\bold{Horocycle flow.}
Motivated by the study of orbits of the horocycle flow, an aspect previously unexplored in the context of the BCZ map, Athreya-Cheung \cite{At} demonstrated that the BCZ map can be used to describe the first return map on the Poincar\'{e} section
$$\Omega':=\{\Lambda_{a,b}:=p_{a,b}SL(2,\mathbb{Z})\in\mathbb{R}^2| a,b\in(0,1], a+b>1\}$$
where $$p_{a,b}=
\begin{pmatrix}
a & b\\
0 & a^{-1}
\end{pmatrix}
$$
for the horocycle flow
$$ h_s=
\begin{pmatrix}
1 & 0\\
-s & 1
\end{pmatrix}
:s\in\mathbb{R}$$
on the space of unimodular lattices $X_2=SL(2,\mathbb{R})/SL(2,\mathbb{Z})$ in $R^2$ (\cite{At}, Theorem 1.1). The first return map $T:\Omega\to\Omega$ defined implicitly by
$$\Lambda_{T(a,b)}=h_{R(a,b)}\Lambda_{a,b},$$
where the first return time
$$R(a,b)=\frac{1}{ab}$$
is given explicitly by the BCZ map.

We call $k(a,b):=\left[\frac{1+a}{b}\right]$ the itinerary function (or index function) of $(a,b)$.

Athreya-Cheung demonstrated the BCZ map is an ergodic, zero-entropy map with respect to the Lebesgue probability measure $dm = 2dadb$ (\cite{At}, Theorem 1.2). Furthermore, they established that $f_{N,I}$ weakly converge to the Lebesgue probability measure $dm=2dadb$ in the Farey triangle where $f_{N,I}=\frac{1}{A_I(N)}\sum\limits_{i:\rho_i\in I}\delta_{T^i\left(\frac1Q,1\right)}$, $I=[\alpha,\beta]\subset[0,1]$ and $A_I(N):=|\mathcal{F}(n)\cap I|$ (\cite{At}, Theorem 1.3). This is proved by Kargaev-Zhigljavsky \cite{Ka} using different methods in the case of $I=[0,1]$. And it can be deduced from Theorem 6 of \cite{Ma}. This can also be derived from the well-known equidistribution principle for closed horocycle on $X_2=SL(2,\mathbb{R})/SL(2,\mathbb{Z})$ by using the projection map. The equidistribution principle was established by Sarnak \cite{Sa} and Eskin-McMullen \cite{Es} for $[\alpha,\beta]=[0,1]$. Hejhal \cite{He96} proved it for any fixed $0<\alpha<\beta<1$. Stronger results where $\beta-\alpha$ is permitted to tend to zero with $N$ have been obtained by Hejhal \cite{He00} and Str\"ombergsson \cite{St}. Athreya-Cheung \cite{At} apply this result to reproduce a finding of Hall \cite{Ha70}. We utilize this result to derive Theorem \ref{106}.

\bold{Riemann hypothesis.}
The significance of the BCZ map is that it introduces a well-defined dynamic system that has a strong connection to the RH. So the investigation of this dynamical system may help the study of the RH.

Zagier \cite{Za} showed that proving an optimal rate of equidistribution for long periodic trajectories for horocycle flow on $X_2$ (that is, an optimal error term in Sarnak's theorem \cite{Sa}) is equivalent to the classical Riemann hypothesis.

For the Farey sequence of order $n$:
$\rho_0=0<\rho_1=\frac{1}{n}<\rho_2=\frac{1}{n-1}<\rho_3<\rho_4<\cdots<\rho_{A_n-1}=\frac{n-1}{n}<\rho_{A_n}=1$, we extend the sequence by setting $\rho_i=\rho_{A_n+i}$ for all $i\in\mathbb{Z}$ and let $\rho_i=\frac{p_i}{q_i}$ where $p_i,q_i\in\mathbb{N}$ and $(p_i,q_i)=1$.

Boca-Cobeli-Zaharescu \cite{Bo} showed that
\begin{llll}  \label{52}
$$T\left(\frac{q_{k}}{n},\frac{q_{k+1}}{n}\right)=\left(\frac{q_{k+1}}{n},\frac{q_{k+2}}{n}\right)$$
holds for all $k\in\mathbb{Z}$.
\end{llll}
Therefore, $T^k\left(\frac{1}{n},1\right)=\left(\frac{q_{k}}{n},\frac{q_{k+1}}{n}\right)$, which implies $\left(\frac{1}{n},1\right)$'s orbit is periodic with period $A_n$.

Using Lemma \ref{52}, the RH equavalence (\ref{40}) can be transformed to
\begin{equation}
\frac{1}{A_n}\sum\limits_{i=1}^{A_n}\left|\chi^{(n)}\left(i,\left(\frac{1}{n},1\right)\right)\right|=O\left(n^{\frac{1}{2}+\epsilon}\right),   \label{41}
\end{equation}
where $\chi^{(n)}\left(i,\left(\frac{1}{n},1\right)\right):=\displaystyle\sum_{j=1}^{i}\left(R\left(T^{j-1}(\frac{1}{n},1)\right)-\frac{n^2}{A_n}\right)$ is a BCZ cocycle. (Details in \S\ref{96}.)

However, directly proving (\ref{41}) is quite challenging. Firstly, $R$ is an unbounded function, so the increment of the cocycle $R\left(T^{i-1}\left(\frac{1}{n},1\right)\right)-\frac{n^2}{A_n}$ can be very large, making the whole process very volatile. Secondly, we observe that when the x-coordinate of $T^i(\frac{1}{n},1)$ is very small, the absolute value of BCZ cocycle term $\chi^{(n)}\left(i,\left(\frac{1}{n},1\right)\right)$ tends to be relatively small, which we refer to as the reset term. However, it is very hard to give any good control of the reset term of the BCZ cocycle, which could otherwise serve as the foundation for a good control of the whole process. If we adopt the classical approach (M\"obius function), we would only have $\left|\chi^{(n)}\left(i,\left(\frac{1}{n},1\right)\right)\right|=O(n\log q_i)$, which is not good enough for our estimation, as it would only yield $O(n^{1+\epsilon})$ for the left-hand side of (\ref{41}).

\bold{Discretized analog of the RH.}
Instead, we can construct the BCZ cocycle's discrete approximation and prove the discretized analog of the RH, which brings us another path to approach the RH.
Let
$$\hat{k}(a,b):=\frac{k(a,b)+k^T(a,b)}{2},$$
where $k(a,b)=\left[\frac{1+a}{b}\right]$ is the itinerary function of the Farey triangle and $k^T(a,b)=k\left(T^{-1}(a,b)\right)=k(b,a)$.

We can replace $R-\frac{n^2}{A_n}$ with $\hat{k}-3$ in (\ref{41}) and prove the discretized analog of the RH holds in a stronger sense.
\begin{ttt}  (main result)  \label{55}

For $n\geq1$, $i\geq1$, let $\theta'_i=\hat{k}\left(T^{i-1}(\frac{1}{n},1)\right)-3$ and $\theta_i=\sum\limits_{j=1}^i\theta'_j$. Then we have
\begin{equation}
\frac{1}{A_n}\sum\limits_{i=1}^{A_n}\left|\theta_i\right|=O(n^{\epsilon}). \label{42}
\end{equation}
\end{ttt}

The reason why we call (\ref{42}) the discretized analog of (\ref{41}) is that $\hat{k}-3$ is very close to $R-\frac{n^2}{A_n}$. Although both functions are unbounded, their difference, however, is bounded! To be specific, we have $-1\leq R-\hat{k}<2$ (Lemma \ref{54}). Walfisz \cite{Wa} showed that
$$A_n=\frac{3}{\pi^2}n^2+O\left(n(\log n)^{\frac23}(\log\log n)^{\frac43}\right),$$
so $\frac{n^2}{A_n}\longrightarrow \frac{\pi^2}{3}$, which is also the average value of $R$ in the Farey triangle. Thus, we know that the difference of $R-\frac{n^2}{A_n}$ and $\hat{k}-3$ is bounded.

This exciting result suggests we can try to approach the RH using an approximation method, which possibly represents a new path towards the RH. We will discuss this idea in \S\ref{101} and \S\ref{32877} when we summarize the result and raise some possible questions for further research from the perspective of function analysis.

\bold{Proof strategy.}
There is one kind of BCZ orbit we are particularly interested in, namely the excursion (Definition \ref{74}). An excursion is a section of BCZ orbit where the x-coordinate of the middle point exceeds those of both endpoints. An interesting fact is that, for any two positive numbers $a,b$ in $(0,1]^2$, there exists a unique excursion where the x-coordinates of the starting point and the ending point are $a$ and $b$, respectively. Therefore, we can define the moduli space of the excursion as $\Xi:(0,1]^2$ (Definition \ref{63}). To prove this property, we provide a detailed description of all points comprising an excursion (Remark \ref{64}). Additionally, we give the asymptotic formula for the length of an excursion (Lemma \ref{65}).

It is noteworthy that the left-hand sides of (\ref{41}) and (\ref{42}) both stem from the periodic orbit of order $n$, which itself constitutes an excursion corresponding to $\left(\frac{1}{n},\frac{1}{n}\right)$. We can generalize these concepts for any function $f$ defined on the Farey triangle and for any excursion, ultimately arriving at the energy function $E(f;a,b)$ (Definition \ref{68}).

The RH equivalence (\ref{41}) is equivalent to
$$E\left(R-\frac{n^2}{A_n};\frac{1}{n},\frac{1}{n}\right)=O\left(n^{\frac{5}{2}+\epsilon}\right),$$
while the discretized analog (\ref{42}) is equivalent to
\begin{equation}
E\left(\hat{k}-3;\frac{1}{n},\frac{1}{n}\right)=O\left(n^{2+\epsilon}\right). \label{432478}
\end{equation}

We prove the main theorem by obtaining a more generalized version of the result (Theorem \ref{87})
$$E|_{\Delta}(\hat{k}-3;a,b)=O\left(\frac{1}{(ab)^{1+\epsilon}}\right),$$
where $\Delta$ represents the golden ratio area of $\Xi$ (Defined in (\ref{7986969})). We use induction on the energy function on the moduli space of excursions to complete the proof of this generalized result. Subsequently, we select $(a,b)=\left(\frac{1}{n},\frac{1}{n}\right)$ to finalize the proof of (\ref{432478}), which is equivalent to the main result Theorem \ref{55}.

\bold{Main tools.}
We define the generalized arithmetic sequence in Definition \ref{58}, which is a sequence where any term can divide the sum of two neighboring terms. The uniqueness of the generalized arithmetic sequence lies in its numerous interesting properties including closure (Lemma \ref{59}) and identity (Theorem \ref{60}).
\begin{itemize}
\item Closure: If a local maximum is eliminated from any generalized arithmetic sequence, it will retain its generalized arithmetic sequence status.
\item Identity: Regarding function $h$, which is defined on the space of generalized arithmetic sequences (Definition \ref{61}), eliminating a local maximum from the sequence will not alter the value of $h$.
\end{itemize}
The reason why we want to study the generalized arithmetic sequence is that the x-coordinates (or y-coordinates) of any BCZ orbit form a generalized arithmetic sequence(Remark \ref{62}). According to the definition of the BCZ map, the y-coordinate of every point in the orbit is equal to the x-coordinate of the subsequent point, which means that all the information about a BCZ orbit can be deduced solely from the x-coordinates or y-coordinates of the orbit. In Remark \ref{62}, we show some examples of generalized arithmetic sequences, including two derived from BCZ orbits and one derived from negative continued fraction.

The reason why we can prove the discretized analog of the RH is that we have good control over $\hat{k}-3$. In fact, for an excursion $(a_j,b_j)_{j=0}^s$, let $\zeta_i=\sum\limits_{j=0}^{i-1}\left(\hat{k}(a_j,b_j)-3\right)$. We have
\begin{itemize}
\item Reset control: $\zeta_s$ can be controlled by the ratio of the x-coordinates of two endpoints of the excursion. (Theorem \ref{66})
\item Overall monotonicity: For $1<i<s-1$, we have $\zeta_1>\zeta_i>\zeta_{s-1}$. (Theorem \ref{84})
\end{itemize}

Obtaining Theorem \ref{55} makes us wonder if there are other functions that could satisfy similar asymptotic formulas which would benefit the possible approximation process to the RH. For $\hat{k}-3$, it satisfies some conditions which lead to Theorem \ref{55}, so we can relax some conditions to get more generalized results. (see Theorem \ref{56} and Theorem \ref{57}) We also provide a family of functions that satisfy the condition of Theorem \ref{56} as an example.

\subsection{Plan of paper}
The remainder of the introduction is about the reformulation of the RH in terms of estimates of $L^1$-averages of BCZ cocycle along periodic orbits of the BCZ map. In \S2, we define the generalized arithmetic sequence (Definition \ref{58}). We establish some of its properties, including closure (Lemma \ref{59}), identity (Theorem \ref{76}, Theorem \ref{60}), and show some examples (Remark \ref{62}). In \S3, we give some basic properties of the $\hat{k}-3$ and explain why it is reasonable to use it as the discrete approximation. In \S4, we introduce the key technical device, the notion of the excursion of the BCZ map, and the moduli space of excursions (Definition \ref{63}). We also establish the main technical results, including the reset control (Theorem \ref{66}) and overall monotonicity (Theorem \ref{84}). In \S5, we define the energy function of an excursion (Definition \ref{68}) and prove the main result by using the induction on the energy function on the moduli space of excursions. In \S6, we obtain the sufficient conditions for the asymptotic formula in the general case (Theorem \ref{56}, Theorem \ref{57}). In \S7 and \S8, we summarize the result in the function analysis point of view and share some ideas, raise some questions for further study.

\subsection{Reformulation of the RH}
\label{96}
We mentioned that (\ref{40}) can be transformed to (\ref{41}). Here's the step-by-step process.

$(\rho_i)_{i=0}^{A_n}$ is the Farey sequence of order $n$. Let $\rho_i=\frac{p_i}{q_i}$ where $p_i, q_i\in\mathbb{N}$ and $(p_i,q_i)=1$. Define $\eta_i=\rho_i-\frac{i}{A_n}$ and $\eta_i'=\frac{1}{q_{i-1}q_i}-\frac{1}{A_n}$ for $1\leq i\leq A_n$. Since $\rho_i=\sum\limits_{j=1}^i(\rho_j-\rho_{j-1})=\sum\limits_{j=1}^i\frac{1}{q_{j-1}q_j}$, thus
$$\eta_i=\rho_i-\frac{i}{A_n}=\sum\limits_{j=1}^i\frac{1}{q_{j-1}q_j}-\frac{i}{A_n}=\sum\limits_{j=1}^i\eta_j',$$
therefore, the RH is equivalent to
$$\sum\limits_{i=1}^{A_n}|\eta_i|=\sum\limits_{i=1}^{A_n}\left|\sum\limits_{j=1}^i\eta_j'\right|=O\left(n^{\frac{1}{2}+\epsilon}\right).$$
Introducing $\iota_i=n^2\eta_i=n^2(\rho_i-\frac{i}{A_n})$ and $\iota_i'=n^2\eta_i'=\frac{n^2}{q_{i-1}q_i}-\frac{n^2}{A_n}$ for $1\leq i\leq A_n$. Then
$$\iota_i=n^2\eta_i=n^2\sum\limits_{j=1}^i\eta_i'=\sum\limits_{j=1}^i\iota_j',$$
therefore, the RH is equivalent to
\begin{equation}
\sum\limits_{i=1}^{A_n}|\iota_i|=O\left(n^{\frac{5}{2}+\epsilon}\right)     \label{43}
\end{equation}
or
$$\frac{1}{n^2}\sum\limits_{i=1}^{A_n}|\iota_i|=\frac{1}{n^2}\sum\limits_{i=1}^{A_n}\left|\sum\limits_{j=1}^i\iota_j'\right|=O\left(n^{\frac{1}{2}+\epsilon}\right).$$
Given that $A_n=\frac{3}{\pi^2}n^2+O\left(n\log n\right)$, the RH is also equivalent to
$$\frac{1}{A_n}\sum\limits_{i=1}^{A_n}|\iota_i|=O\left(n^{\frac{1}{2}+\epsilon}\right). $$
By lemma \ref{52}, $T^i\left(\frac{1}{n},1\right)=T^i\left(\frac{q_0}{n},\frac{q_1}{n}\right)=\left(\frac{q_i}{n},\frac{q_{i+1}}{n}\right)$. Thus, $R(T^i(\frac{1}{n},1))=\frac{n^2}{q_iq_{i+1}}$, which implies that
$$\iota_i'=R\left(T^{i-1}\left(\frac{1}{n},1\right)\right)-\frac{n^2}{A_n}. $$
Since $A_n=\frac{3}{\pi^2}n^2+O(n\log n)$, $\frac{n^2}{A_n}\rightarrow\frac{\pi^2}{3}$ as $n\rightarrow+\infty$. Consequently, when $n$ is very large, $\iota_i'\approx R\left(T^{i-1}\left(\frac{1}{n},1\right)\right)-\frac{\pi^2}{3}$.

BCZ cocycle $\chi^{(n)}:\mathbb{Z}^+\times\Omega\rightarrow\mathbb{R}$ can be defined over $T$:
$$\chi^{(n)}\left(i,(a,b)\right):=\sum_{j=1}^{i}\left(R\left(T^{j-1}(a,b)\right)-\frac{n^2}{A_n}\right), $$
which satisfies the cocycle property:
$$\chi^{(n)}\left(i_1+i_2,(a,b)\right)=\chi^{(n)}\left(i_2,T^{i_1}(a,b)\right)+\chi^{(n)}\left(i_1,(a,b)\right). $$
Moreover, we have
\begin{equation}
\chi^{(n)}\left(i,\left(\frac{1}{n},1\right)\right)=\sum_{j=1}^{i}\left(R\left(T^{j-1}(\frac{1}{n},1)\right)-\frac{n^2}{A_n}\right)=\sum_{j=1}^{i}\iota_{j}'=\iota_i. \label{48}
\end{equation}
Therefore, according to (\ref{43}), the RH is equivalent to
$$\sum\limits_{i=1}^{A_n}\left|\chi^{(n)}\left(i,\left(\frac{1}{n},1\right)\right)\right|=O\left(n^{\frac{5}{2}+\epsilon}\right)$$
or
$$\frac{1}{A_n}\sum\limits_{i=1}^{A_n}\left|\chi^{(n)}\left(i,\left(\frac{1}{n},1\right)\right)\right|=O\left(n^{\frac{1}{2}+\epsilon}\right). $$

We mentioned that $\hat{k}$ is very close to $R$. In fact, we have the following inequality:
\begin{llll}    \label{54}
$$-1\leq R-\hat{k}<2.$$

To be specific:
\begin{itemize}
\item $-\frac{1}{2}\leq R-\hat{k}<\frac{3}{2}+\frac{2}{n}$ when $(k,k^T)=(n,1)$ or $(1,n)$, $n\geq4$.
\item $\frac{4}{15}\leq R-\hat{k}<\frac{3}{2}$ when $(k,k^T)=(4,2)$ or $(2,4)$.
\item $-\frac{1}{2}\leq R-\hat{k}<\frac{19}{15}$ when $(k,k^T)=(3,1)$ or $(1,3)$.
\item $-\frac{5}{12}\leq R-\hat{k}<\frac{5}{3}$ when $(k,k^T)=(3,2)$ or $(2,3)$.
\item $-\frac{1}{2}< R-\hat{k}<\frac{7}{12}$ when $(k,k^T)=(2,1)$ or $(1,2)$.
\item $-1\leq R-\hat{k}<2$ when $(k,k^T)=(2,2)$.\\
\end{itemize}

\end{llll}

\section{Generalized arithmetic sequence}
\label{102}
In this section, we will introduce the generalized arithmetic sequence and present its basic properties. These include closure (Lemma \ref{59}), identity (Theorem \ref{76}, Theorem \ref{60}), and some examples (Remark \ref{62}).

\begin{ddd}      \label{58}
Let $(a_i)_{i=1}^n$ (resp. $(a_i)_{i=1}^{+\infty}$ and $(a_i)_{i=-\infty}^{+\infty}$) be a sequence of positive real numbers.

We call it a generalized arithmetic sequence when
$$a_i\mid a_{i-1}+a_{i+1} $$
for $i\in[2,n-1]$ (resp. $i\in[2,+\infty)$ and $i\in\mathbb{Z}$).

We call $(a_i)_{i=1}^n$ a cyclic generalized arithmetic sequence if
\begin{itemize}
\item $n=1$
\item $a_1\mid a_2+a_2$, $a_2\mid a_1+a_1$ when $n=2$
\item $a_i\mid a_{i-1}+a_{i+1}$ for $i\in[2,n-1]$, $a_1\mid a_2+a_n$, $a_n\mid a_{n-1}+a_1$ when $n\geq3$
\end{itemize}

\end{ddd}

\begin{rrr}

$(a_i)_{i=1}^n$ being a cyclic generalized arithmetic sequence is equivalent to a periodic sequence $(a_i)_{i=-\infty}^{+\infty}$ being a generalized arithmetic sequence, where $a_i=a_{n+i}$ for $i\in\mathbb{Z}$.

For simplicity, when discussing the cyclic generalized arithmetic sequence, we let $a_i=a_{n+i}$ for $i\in\mathbb{Z}$. For example, we use $a_{-1},a_{n+1}$ to refer $a_{n-1},a_1$.
\end{rrr}
According to example 1 of Remark \ref{62}, we know that the x-coordinate of any section of BCZ orbit: $(x_0,x_1,x_2,\cdots,x_n)$ is a generalized arithmetic sequence. By the definition of the Farey triangle, we have $1<x_i+y_i=x_i+x_{i+1}$, $x_i\in(0,1]$, therefore, we can give this special kind of generalized arithmetic sequence a new name: BCZ sequence.

Next, we demonstrate the closure property of the generalized arithmetic sequence, that is, if we eliminate a local maximum from a (resp. cyclic) generalized arithmetic sequence, it will still be a (resp. cyclic) generalized arithmetic sequence.

\begin{llll}(Local maximum)

If $a_m$ is a local maximum of a (resp. cyclic) generalized arithmetic sequence $(a_i)_{i=1}^n$, which means that $a_m>a_{m-1},a_{m+1}$ where (resp. $m\in[1,n]$) $m\in[2,n-1]$, then we have
$$a_m=a_{m-1}+a_{m+1}. $$
This also holds for $(a_i)_{i=1}^{+\infty}$ when $m>1$ or $(a_i)_{i=-\infty}^{+\infty}$ when $m\in\mathbb{Z}$.

\end{llll}

\begin{llll}    \label{80}
Let $(a_i)_{i=1}^n$ be a (resp. cyclic) generalized arithmetic sequence, and let $a_m$
be a largest term among them (i.e., $a_m\geq a_i$ for all $i\in[1,n]$), where (resp. $m\in[1,n]$) $m\in[2,n-1]$. If there exists $j$ such that $a_m>a_j$, then $a_m$ is a local maximum.

This also holds for $(a_i)_{i=1}^{+\infty}$ when $m>1$ or $(a_i)_{i=-\infty}^{+\infty}$ when $m\in\mathbb{Z}$.

\end{llll}

\begin{proof}[proof]
Without loss of generality, we assume that $j>m$. Since $a_m\geq a_{m-1},a_{m+1}$ while $a_m\mid a_{m-1}+a_{m+1}$, we know that $a_m\leq a_{m-1}+a_{m+1}\leq 2a_m$. If $a_{m-1}+a_{m+1}=2a_m$, then $a_{m-1}=a_m=a_{m+1}$. Since there exists $a_j<a_m$, we can pick the smallest $t>m+1$ such that $a_t<a_m$. Then $a_{t-1}<a_{t-2}+a_t<2a_{t-1}$, which contradicts the fact that $a_{t-1}\mid a_{t-2}+a_t$. Thus, we have $a_{m-1}+a_{m+1}=a_m$.

\end{proof}

\begin{llll}     \label{59}    (Closure of the space of generalized arithmetic sequences)

Let $(a_i)_{i=1}^n$ be a (resp. cyclic) generalized arithmetic sequence, and let $a_m$ be a local maximum. If (resp. $n\geq2$) $n\geq3$, then after eliminating $a_m$ from the sequence, the new sequence
$$(a_1,a_2,\cdots,a_{m-1},a_{m+1},\cdots,a_n)$$
is still a (resp. cyclic) generalized arithmetic sequence.

This also holds for $(a_i)_{i=1}^{+\infty}$ when $m>1$ or $(a_i)_{i=-\infty}^{+\infty}$ when $m\in\mathbb{Z}$.

\end{llll}

\begin{proof}[proof]
When $(a_i)_{i=1}^n$ is a cyclic generalized arithmetic sequence, since $a_m=a_{m-1}+a_{m+1}$, we have $a_{m-1}+a_{m+2}=a_m-a_{m+1}+a_{m+2}$. Since $a_{m+1}\mid a_m+a_{m+2}$, it follows that $a_{m+1}\mid a_{m-1}+a_{m+2}$. By symmetry, we also know that $a_{m-1}\mid a_{m-2}+a_{m+1}$.

When $(a_i)_{i=1}^n$ is a generalized arithmetic sequence, we only need to prove the statement when $n\geq4$. If $m\leq n-2$, by the above, we know that $a_{m+1}\mid a_{m-1}+a_{m+2}$. If $m\geq 3$, by the above, we know that $a_{m-1}\mid a_{m-2}+a_{m+1}$.

The proof for the infinity sequence case is similar .

\end{proof}

Next, we want to give the definition of the itinerary sequence, which is inspired by the itinerary of the BCZ orbit.

\begin{ddd} \label{772}
For a generalized arithmetic sequence $(a_i)_{i=1}^n$(resp. $(a_i)_{i=1}^{+\infty}$, $(a_i)_{i=-\infty}^{+\infty}$, and for a cyclic generalized arithmetic sequence $(a_i)_{i=1}^n$), let its itinerary sequence be $(k_i)_{i=2}^{n-1}$(resp. $(k_i)_{i=2}^{+\infty}$, $(k_i)_{i=-\infty}^{+\infty}$ and $(k_i)_{i=1}^n$) where
$$k_i:=\frac{a_{i-1}+a_{i+1}}{a_i}. $$

\end{ddd}

Next, we will define two functions $h$ and $\hat{h}$. $h$ is essential to our proofs in the following sections, the motivation for its definition can be found in Lemma \ref{72}. We will use $\hat{h}$ to prove an important result (Theorem \ref{73}). We will also demonstrate the identity property, which means that if we eliminate a local maximum from a (resp. cyclic) generalized arithmetic sequence, not only will it still be a (resp. cyclic) generalized arithmetic sequence, but also the value of $h$ (resp. $\hat{h}$) at the sequence would remain the same.

\begin{ddd}    \label{61}
$(a_i)_{i=1}^n$ is a sequence where $n\geq4$. We define $h$ as
$$h(a_1,a_2,\cdots,a_n):=\sum_{i=2}^{n-2}\left(\frac{\frac{a_i+a_{i+2}}{a_{i+1}}+\frac{a_{i-1}+a_{i+1}}{a_i}}{2}-3\right). $$

\end{ddd}

The function $h$ applied to a sequence of length $n$ is the sum of $h$ applied to all $n-3$ consecutive 4-tuples of this sequence.

When $(a_i)_{i=1}^n$ is a generalized arithmetic sequence, we have
$$h(a_1,a_2,\cdots,a_n)=\sum_{i=2}^{n-2}\left(\frac{k_i+k_{i+1}}{2}-3\right)=\sum_{i=2}^{n-1}(k_i-3)-\left(\frac{k_2+k_{n-1}}{2}-3\right), $$
which means that the value of $h$ at a generalized arithmetic sequence can be determined by its itinerary sequence.

\begin{ddd}
$(a_i)_{i=1}^n$ is a sequence where $n\geq1$. Let $a_{i+n}=a_i$ for $i\in\mathbb{Z}$. We define $\hat{h}$ as
$$\hat{h}(a_1,a_2,\cdots,a_n):=\sum_{i=1}^n\left(\frac{\frac{a_i+a_{i+2}}{a_{i+1}}+\frac{a_{i-1}+a_{i+1}}{a_i}}{2}-3\right).  $$
\end{ddd}

The function $\hat{h}$ applied to a sequence of length $n$ is the sum of $\hat{h}$ applied to all $n$ consecutive 4-tuples of this sequence if we consider this sequence to be cyclic.

When $(a_i)_{i=1}^n$ is a cyclic generalized arithmetic sequence, we have
$$h(a_1,a_2,\cdots,a_n)=\sum_{i=1}^n\left(\frac{k_i+k_{i+1}}{2}-3\right)=\sum_{i=1}^n(k_i-3), $$
which means that the value of $\hat{h}$ at a cyclic generalized arithmetic sequence can also be determined by its itinerary sequence.

\begin{ttt}   \label{76}     (Identity of $\hat{h}$)

$(a_i)_{i=1}^n$ is a cyclic generalized arithmetic sequence where $n\geq2$. $a_m$ is a local maximum. We have
$$\hat{h}(a_1,a_2,\cdots,a_{m-1},a_{m+1},\cdots,a_n)=\hat{h}(a_1,a_2,\cdots,a_n). $$

\end{ttt}

\begin{ttt}     \label{60}    (Identity of $h$)

Let $(a_i)_{i=1}^n$ be a generalized arithmetic sequence where $n\geq7$, and let $a_m$ be a local maximum where $m\in[4,n-3]$. We have
$$h(a_1,a_2,\cdots,a_{m-1},a_{m+1},\cdots,a_n)=h(a_1,a_2,\cdots,a_n). $$

\end{ttt}

Now, we have already shown the closure and identity. We now explain their application. As we mentioned before, the x-coordinates or the y-coordinates of the BCZ orbit are not only a generalized arithmetic sequence but also a BCZ sequence. By eliminating every local maximum, we get a new generalized arithmetic sequence. However, the value of $h$ at the sequence would remain the same. So after the process of elimination:
$$\text{BCZ sequence}\longrightarrow \text{generalized arithmetic sequence}\longrightarrow \cdots$$$$\longrightarrow \text{generalized arithmetic sequence in the simplest form}.$$
The value of $h$ at the original BCZ sequence is actually equal to $h$ at the generalized arithmetic sequence in the simplest form. Oftentimes, the latter is easy to calculate while the former is not. So, by this process, we can compute $h$ at a long, complicated generalized arithmetic sequence by calculating the $h$ at the generalized arithmetic sequence derived from the original, complicated one by eliminating local maximum until it can't. This same method can also be applied to the cyclic generalized arithmetic sequence and $\hat{h}$. We will use the method in the proof of Theorem \ref{73} and Theorem \ref{66}.

Now, after presenting the properties of the (resp. cyclic) generalized arithmetic sequence, we show some examples of the (resp. cyclic) generalized arithmetic sequence.

\begin{rrr}    \label{62}     (Examples of (resp. cyclic) generalized arithmetic sequence)

1. $T$ is the BCZ map, $(a,b)$ is a point of the Farey triangle. $(a_i, b_i)=T^i(a,b)$ for $i\in\mathbb{Z}$. Then for any $c\in\mathbb{R}^+, m\in\mathbb{Z}, n\in\mathbb{N}$, $(ca_i)_{i=m}^{m+n}$ is a generalized arithmetic sequence.

The reason is that for $i\in[m+1,m+n-1]$, by the definition of the BCZ map,
$$a_{i+1}=b_i=k(a_{i-1},b_{i-1})b_{i-1}-a_{i-1}=k(a_{i-1},b_{i-1})a_i-a_{i-1}, $$
thus
$$a_i\mid a_{i-1}+a_{i+1}, $$
$$ca_i\mid ca_{i-1}+ca_{i+1}. $$

2. $(a,b)$ is a rational point of the Farey triangle which means that $\frac{b}{a}$ is a rational number. Let $\frac{b}{a}=\frac{p}{q}$ where $(p,q)=1$, then let $l=\frac{q}{a}=\frac{p}{b}$, so $(a,b)=(\frac{q}{l},\frac{p}{l})$. Let $n=[l]$, $A_n$ be the length of the Farey sequence of order $n$. By Lemma \ref{52}, we know that $\forall m\in\mathbb{Z}$,
$$T^m\left(\frac{1}{n},1\right)=\left(\frac{q_m}{n},\frac{q_{m+1}}{n}\right), $$
where $(\frac{p_i}{q_i})_{i=0}^{A_n}$ is the Farey sequence of order $n$ and we let $\frac{p_i}{q_i}=\frac{p_{i+A_n}}{q_{i+A_n}}$ for $\forall i\in\mathbb{Z}$.
Then, we can deduce that $\forall m\in\mathbb{Z}$,
$$T^m\left(\frac{1}{l},\frac{n}{l}\right)=\left(\frac{q_m}{l},\frac{q_{m+1}}{l}\right).$$
Since $(q,p)=1$, $p+q>l\geq n$, thus $(q,p)$ are the denominators of neighboring terms of the Farey sequence of order $n$, we know there exists a unique $k\in[0,A_n-1]$ such that $(q,p)=(q_k,q_{k+1})$. Therefore,
$$(a_m,b_m)=T^m(a,b)=T^m\left(\frac{q_k}{l},\frac{q_{k+1}}{l}\right)=\left(\frac{q_{k+m}}{l},\frac{q_{k+m+1}}{l}\right).$$
Then we know that for any $m\in\mathbb{Z}$,
$$T^{A_n}(a_m,b_m)=(a_m,b_m).  $$
So for any $c\in\mathbb{R}^+$, $m\in\mathbb{Z}$, $j\in\mathbb{Z}^+$, $(ca_i)_{i=m}^{m+jA_n-1}$ is a cyclic generalized arithmetic sequence.

The reason is the same as we mentioned above in 1. We have $ca_i\mid ca_{i-1}+ca_{i+1}$. Since $a_m=a_{m+jA_n}$, $a_{m+jA_n-1}=a_{m-1}$, so we also have
$ca_m\mid ca_{m+jA_n-1}+ca_{m+1}$, $ca_{m+jA_n-1}\mid ca_{m+jA_n-2}+ca_m$.\\

3. Let $[b_0;b_1,b_2,\cdots,b_n]_{-}$ be the negative continued fraction, which means that
$$[b_0;b_1,b_2,\cdots,b_n]_{-}=b_0-\frac{1}{b_1-\frac{1}{\ddots-\frac{1}{b_n}}},  $$
where $b_i\geq 2$ for $i>0$.

For an irrational number $\alpha>0$, let $\alpha=[b_0;b_1,b_2,\cdots]_{-}$, $\frac{r_n}{s_n}=[b_0;b_1,b_2,\cdots,b_n]_{-}$ where $(r_n,s_n)=1$, $r_n,s_n\in\mathbb{Z}^+$ for $n\geq0$. Then we have
$$r_{n+2}=b_{n+2}r_{n+1}-r_n, $$
$$s_{n+2}=b_{n+2}s_{n+1}-s_n,  $$
for $n\geq0$, which means that $(r_n)_{n=0}^{+\infty}$ and $(s_n)_{n=0}^{+\infty}$ are both generalized arithmetic sequences.

The itinerary sequence of these two generalized arithmetic sequences is $(b_i)_{i=2}^{+\infty}$. We have the following correspondence between $(b_i)_{i=0}^{+\infty}$ and
$(a_i)_{i=0}^{+\infty}$ where $[a_0;a_1,a_2,\cdots]$ is the continued fraction of $\alpha$. Let $c_i=\sum_{j=1}^ia_{2j-1}$, we have
$$b_0=a_0+1;  $$
$$b_{c_i}=a_{2i}+2,  $$
for $i\geq1$;
$$b_k=2, $$
for $k\neq c_i$.

Let $\frac{p_n}{q_n}=[a_0;a_1,a_2,\cdots,a_n]$, we know that $\frac{p_n}{q_n}$ is a best approximation of $\alpha$. Meanwhile, we can deduce that $\frac{r_n}{s_n}$ is also a best approximation of $\alpha$. And we also have the correspondence between $\frac{p_n}{q_n}$ and $\frac{r_n}{s_n}$:
$$\frac{p_{2i-1}}{q_{2i-1}}=\frac{r_{c_i-1}}{s_{c_i-1}}.  $$

\end{rrr}

\begin{proof}[proof](of Theorem \ref{76})

For $n=2$, assume that $m=2$, then $a_2=2a_1$, so
$$\hat{h}(a_1,a_2)=-1=\hat{h}(a_1).$$

For $n=3$, assume that $m=3$, then $a_1+a_2=a_3$, assume that $a_1\leq a_2$, then $a_3\leq 2a_2$, $a_2<a_1+a_3\leq3a_2$.

When $a_1+a_3=2a_2$, we have $a_2=2a_1$, $a_3=3a_1$, so
$$\hat{h}(a_1,a_2,a_3)=-1=\hat{h}(a_1,a_2).$$
When $a_1+a_3=3a_2$, we have $a_2=a_1$, $a_3=2a_1$, so
$$\hat{h}(a_1,a_2,a_3)=-2=\hat{h}(a_1,a_2).$$

For $n\geq4$,
\begin{align*}
&\hat{h}(a_1,a_2,\cdots,a_n)-\hat{h}(a_1,a_2,\cdots,a_{m-1},a_{m+1},\cdots,a_n)\\
=&\left(\frac{\frac{a_{m-2}+a_{m}}{a_{m-1}}+\frac{a_{m-3}+a_{m-1}}{a_{m-2}}}{2}-3\right)+
\left(\frac{\frac{a_{m-1}+a_{m+1}}{a_{m}}+\frac{a_{m-2}+a_{m}}{a_{m-1}}}{2}-3\right)\\
&+\left(\frac{\frac{a_{m}+a_{m+2}}{a_{m+1}}+\frac{a_{m-1}+a_{m+1}}{a_{m}}}{2}-3\right)+
\left(\frac{\frac{a_{m+1}+a_{m+3}}{a_{m+2}}+\frac{a_{m}+a_{m+2}}{a_{m+1}}}{2}-3\right)\\
&-\left(\frac{\frac{a_{m-2}+a_{m+1}}{a_{m-1}}+\frac{a_{m-3}+a_{m-1}}{a_{m-2}}}{2}-3\right)-
\left(\frac{\frac{a_{m-1}+a_{m+2}}{a_{m+1}}+\frac{a_{m-2}+a_{m+1}}{a_{m-1}}}{2}-3\right)\\
&-\left(\frac{\frac{a_{m-1}+a_{m+2}}{a_{m+1}}+\frac{a_{m+1}+a_{m+3}}{a_{m+2}}}{2}-3\right)\\
=&\frac{a_m-a_{m+1}}{a_{m-1}}+\frac{a_m-a_{m-1}}{a_{m+1}}+\frac{a_{m-1}+a_{m+1}}{a_{m}}-3.
\end{align*}

Since $a_m=a_{m-1}+a_{m+1}$, we have $\frac{a_{m-1}+a_{m+1}}{a_{m}}=1$, $\frac{a_m-a_{m+1}}{a_{m-1}}=1$, $\frac{a_m-a_{m-1}}{a_{m+1}}=1$. Therefore,
$$\hat{h}(a_1,a_2,\cdots,a_n)-\hat{h}(a_1,a_2,\cdots,a_{m-1},a_{m+1},\cdots,a_n)=0.$$

\end{proof}

\begin{proof}[proof](of Theorem \ref{60})

The proof is almost the same as the proof of Theorem \ref{76}.

\end{proof}

\section{Discrete approximation}
\label{95}
In this section, we will show the basic property of $\hat{k}-3$ and explain why it is reasonable to replace $R-\frac{n^2}{A_n}$ with its discrete approximation $\hat{k}-3$ in (\ref{41}).

For the first return time function $R$, we have
$$\frac{1}{|\Omega|}\int_{\Omega}R dm=\frac{\pi^2}{3},$$
where $dm=2dadb$ is the Lebesgue probability measure in the Farey triangle. So the average of $R$ in the Farey triangle is $\frac{\pi^2}{3}$. As we know from \S\ref{96}, for $\iota'_i=R(T^{i-1}(\frac{1}{n},1))-\frac{n^2}{A_n}$, $\frac{n^2}{A_n}$ acts as an average term which makes $\iota_{A_n}$ go back to zero. When $n$ is very large, the average term is approximate to $\frac{\pi^2}{3}$, which is the average of $R$.

For $\hat{k}$, we want to make a similar statement as (\ref{41}) which will serve as a tool to approximate (\ref{41}). In order to do so, the plan is to construct a sequence to approximate $(\iota_1,\iota_2,\cdots,\iota_{A_n})$.

The most obvious sequence to construct is $(\theta_1,\theta_2,\cdots,\theta_{A_n})$ where $\theta_i=\sum\limits_{j=1}^i\theta'_j$ and $\theta'_i= \hat{k}\left(T^{i-1}\left(\frac{1}{n},1\right)\right)-\lambda$ for some $\lambda\in\mathbb{R}$.

Since
$$\frac{1}{|\Omega|}\int_{\Omega}\hat{k} dm=3,$$
we are going to study the property of $\hat{k}-3$ first.

Firstly, we will calculate $\sum\limits_{i=1}^{A_n}\left(\hat{k}\left(T^{i-1}\left(\frac{1}{n},1\right)\right)-3\right)$.

Let $\rho_i=\frac{p_i}{q_i}$ be the terms of the Farey sequence of order $n$, $q_{i+A_n}=q_i$ for $i\in\mathbb{Z}$. Then by Lemma \ref{52}, we have $T^i\left(1,\frac{1}{n}\right)=\left(\frac{q_{i-1}}{n},\frac{q_i}{n}\right)$ for $i\in\mathbb{Z}$.

\begin{llll}   \label{72}
$$\hat{k}\left(T^i\left(\frac{1}{n},1\right)\right)=\frac{\frac{q_i+q_{i+2}}{q_{i+1}}+\frac{q_{i-1}+q_{i+1}}{q_i}}{2}.$$

\end{llll}

\begin{proof}[proof]

Since $T\left(T^i\left(\frac{1}{n},1\right)\right)=\left(\frac{q_{i+1}}{n},\frac{q_{i+2}}{n}\right)$, by the definition of the BCZ map,
$$\frac{q_{i+2}}{n}=k\left(T^i\left(\frac{1}{n},1\right)\right)\cdot\frac{q_{i+1}}{n}-\frac{q_i}{n},$$
thus
$$k\left(T^i\left(\frac{1}{n},1\right)\right)=\frac{\frac{q_i}{n}+\frac{q_{i+2}}{n}}{\frac{q_{i+1}}{n}}=\frac{q_i+q_{i+2}}{q_{i+1}},$$
for $k^T(T^i(\frac{1}{n},1))$, since $q_i=q_{A_n-i}$, we have
$$k^T\left(T^i\left(\frac{1}{n},1\right)\right)=k^T\left(\frac{q_i}{n},\frac{q_{i+1}}{n}\right)=k\left(\frac{q_{i+1}}{n},\frac{q_i}{n}\right)=k\left(\frac{q_{A_n-i-1}}{n},\frac{q_{A_n-i}}{n}\right),$$
$$=\frac{q_{A_n-i-1}+q_{A_n-i+1}}{q_{A_n-i}}=\frac{q_{i-1}+q_{i+1}}{q_i},$$
therefore
$$\hat{k}\left(T^i\left(\frac{1}{n},1\right)\right)=\frac{\frac{q_i+q_{i+2}}{q_{i+1}}+\frac{q_{i-1}+q_{i+1}}{q_i}}{2}.$$

\end{proof}

\begin{ccc}  \label{79}
$$\sum\limits_{i=1}^{A_n}\left(\hat{k}\left(T^{i-1}\left(\frac{1}{n},1\right)\right)-3\right)=\hat{h}(q_1,q_2,\cdots,q_{A_n}).$$
\end{ccc}

\begin{rrr}   \label{78}
As we mentioned before, by Lemma \ref{52}, we have $T^i(\frac{1}{n},1)=(\frac{q_i}{n},\frac{q_{i+1}}{n})$. By Remark \ref{62}, we know that for any $c\in\mathbb{R}^+$, $m\in\mathbb{Z}$, $j\in\mathbb{Z}^+$, $(ca_i)_{i=m}^{m+jA_n-1}$ is a cyclic generalized arithmetic sequence. Let $c=n$, $m=1$, we know that $(q_i)_{i=1}^{A_n}$ is a cyclic generalized arithmetic sequence.
\end{rrr}

Then we have the following theorem.

\begin{ttt}    \label{73}
$$\sum\limits_{i=1}^{A_n}\left(\hat{k}\left(T^{i-1}\left(\frac{1}{n},1\right)\right)-3\right)=-1.$$
\end{ttt}

\begin{proof}[proof]
By Corollary \ref{79}, we have
$$\sum\limits_{i=1}^{A_n}\left(\hat{k}\left(T^{i-1}\left(\frac{1}{n},1\right)\right)-3\right)=\hat{h}(q_1,q_2,\cdots,q_{A_n}).$$
By Remark \ref{78}, we know that $(q_i)_{i=1}^{A_n}$ is a cyclic generalized arithmetic sequence. So if we pick a largest term $q_m$ from the sequence, since $q_{A_n}$ is the only term among $(q_i)_{i=1}^{A_n}$ that is equal to 1, it follows that $q_m>1=q_{A_n}$. Therefore, by Lemma \ref{80}, $q_m$ is a local maximum. If we eliminate $q_m$ from the sequence, by Theorem \ref{76}, we know that the value of $\hat{h}$ at the sequence will remain the same. Furthermore, by Lemma \ref{59}, the new sequence will still be cyclic generalized arithmetic sequence. Then we can repeat this process of eliminating a largest term until there is only one term left, which would be $q_{A_n}$.
Therefore, we have
$$\hat{h}(q_1,q_2,\cdots,q_{A_n})=\hat{h}(q_{A_n})=-1,$$
which means that
$$\sum\limits_{i=1}^{A_n}\left(\hat{k}\left(T^{i-1}\left(\frac{1}{n},1\right)\right)-3\right)=-1.$$

\end{proof}

So if we let $\lambda=\lambda_1=3$, $\theta_{A_n}=-1$ which would be a decent approximation to $\iota_{A_n}=0$. Let the corresponding sequence $\{\theta_i|i\in[1,A_n]\}$ be $\{\theta_{i,\lambda_1}|i\in[1,A_n]\}$.

By Theorem \ref{73}, we know that
$$\sum\limits_{i=1}^{A_n}\hat{k}\left(T^{i-1}\left(\frac{1}{n},1\right)\right)=3A_n-1,$$
which means that if we let $\lambda=\lambda_2=\frac{3A_n-1}{A_n}$, $\theta_{A_n}=0$. Let the corresponding sequence $\{\theta_i|i\in[1,A_n]\}$ be $\{\theta_{i,\lambda_2}|i\in[1,A_n]\}$.

But it would not change the value of the sequence that much because
$$\theta_{i,\lambda_2}-\theta_{i,\lambda_1}=\frac{i}{A_n},$$
$$\left|\sum\limits_{i=1}^{A_n}\left|\theta_{i,\lambda_2}\right|-\sum\limits_{i=1}^{A_n}\left|\theta_{i,\lambda_1}\right|\right|
\leq\sum\limits_{i=1}^{A_n}|\theta_{i,\lambda_2}-\theta_{i,\lambda_1}|=\sum\limits_{i=1}^{A_n}\frac{i}{A_n}=\frac{A_n+1}{2}.$$

$\frac{A_n+1}{2}=O(n^2)$ is ignorable because in (\ref{43}), the goal is to prove $\sum\limits_{i=1}^{A_n}|\iota_i|=O\left(n^{\frac{5}{2}+\epsilon}\right)$.
However, it is easier to control $\hat{k}-3$ than $\hat{k}-\frac{3A_n-1}{A_n}$, so we are going to use $\hat{k}-3$ moving forward and let $\theta'_i= \hat{k}\left(T^{i-1}\left(\frac{1}{n},1\right)\right)-3$ and $\theta_i=\sum\limits_{j=1}^i\theta'_j$.

\section{Excursion and properties of $\hat{k}-3$}
\label{105}

In this section, we introduce the excursion, show some of its properties, and define the moduli space of excursions. Then we establish the main technical results, including the reset control (Theorem \ref{66}) and overall monotonicity (Theorem \ref{84}), which will explain why the discretized approximation function $\hat{k}-3$ is easier to control.

\subsection{Excursion}

\begin{ddd}   \label{74}   (excursion)

Let the depth of a point $(a,b)$ be $\frac{1}{a}$. If two points of a BCZ orbit are deeper than all the points lying between these two points, we refer to the process from one of these two points to the other as an excursion. For example, let $(a_n,b_n)=T^n(a,b)$, $s<t$, $s,t\in\mathbb{Z}$. If
$$\frac{1}{a_m}<\frac{1}{a_s},\frac{1}{a_t},$$
for any $m\in(s,t)$, then the process that starts from $(a_s,b_s)$ and stops when it reaches $(a_t,b_t)$ for the first time is considered an excursion. In the sequel, for convenience, we will omit the phrase ``for the first time" and simply refer to the process from $(a_s,b_s)$ to $(a_t,b_t)$ as an excursion.

\end{ddd}

Now, we show some examples of excursions. Let $r_i$ be the number such that the Farey sequence terms $\frac{p_{r_i}}{q_{r_i}}=\frac{1}{i}$ for $i\geq1$.

\begin{itemize}
\item For any $j\in[1,r_i-1]$, since $\frac{p_j}{q_j}<\frac{1}{i}$, we know that $q_j>i=q_{r_i}\geq 1=q_0$. As $T^j(\frac{1}{n},1)=(\frac{q_j}{n},\frac{q_{j+1}}{n})$, thus the process from $(\frac{1}{n},1)$ to $(\frac{q_{r_i}}{n},\frac{q_{r_i+1}}{n})$ is an excursion. In particular, when $i=1$, $r_1=A_n$, we know that the process from $(\frac{1}{n},1)$ back to itself is not only a periodic orbit, but also an excursion.

\item For any $j\in[r_{i+1}+1,r_i-1]$, since $\frac{1}{i+1}<\frac{p_j}{q_j}<\frac{1}{i}$, we know that $q_j>i+1=q_{r_{i+1}}> i=q_{r_i}$. Thus the process from $(\frac{q_{r_{i+1}}}{n},\frac{q_{r_{i+1}+1}}{n})$ to $(\frac{q_{r_i}}{n},\frac{q_{r_i+1}}{n})$ is an excursion.
\end{itemize}

In \cite{me}, we show that for an excursion from $(a_0,b_0)$ to $T^s(a_0,b_0)=(a_s,b_s)$, when $i\in[1,s-1]$, we have $a_i=u_ia_0+v_ia_s$, where $u_i,v_i\in\mathbb{Z}^+$, $(u_i,v_i)=1$, when $1\leq i< j \leq s-1$, we have
$$\frac{v_i}{u_i}<\frac{v_j}{u_j}.$$
In fact, we also have the following one-to-one correspondence:
\begin{equation}\{a_i\mid i\in[1,s-1]\}\longleftrightarrow\{(u,v)\mid u,v\in\mathbb{Z}^+,(u,v)=1,ua_0+va_s\leq1\}.\label{64}\end{equation}
The number of those primitive points is $s-1$ which is also the number of pairs of coprime positive integers $(u,v)$ such that $ua_0+va_s\leq1$.

Two points worth mentioning are $a_1=\left[\frac{1-a_s}{a_0}\right]a_0+a_s$ and $a_{s-1}=a_0+\left[\frac{1-a_0}{a_s}\right]a_s$. We also obtain the following estimation of the length of the excursion.

\begin{llll}       \label{65}     (length of the excursion)

For an excursion $(a_i,b_i)_{i=0}^s$ where $s>0$, $a_0=a$, $a_s=b$, we have
$$s=\frac{3}{\pi^2}\frac{1}{ab}+O\left(\max\left\{\frac{1}{a},\frac{1}{b}\right\}\log\left(\min\left\{\frac{1}{a},\frac{1}{b}\right\}\right)\right).$$

\end{llll}

\begin{rrr}(Upper bound for the length of the excursion)

In Lemma \ref{65}, we provided an estimation of the length of an excursion when it is particularly long. However, we can also establish an upper bound for the length of any excursion.
\begin{align*}
s-1&=\sum_{(u,v)=1, \text{ }ua+vb\leq1}1\\
&\leq\sum_{(u,v): \text{ }ua+vb\leq1}1.
\end{align*}

If $a\in\left(\frac{1}{k+1},\frac{1}{k}\right]$ where $k\in\mathbb{Z}^+$, then $[\frac{1}{a}]=k$, thus $[\frac{1}{a}]a>\frac{k}{k+1}\geq\frac{1}{2}$. Also, we have $[\frac{1}{b}]b>\frac{1}{2}$. So $[\frac{1}{a}]a+[\frac{1}{b}]b>1$ which means that $([\frac{1}{a}],[\frac{1}{b}])$ doesn't satisfy $ax+by\leq1$. Then we have
\begin{align*}
s&\leq\left(\sum_{(u,v): \text{ }ua+vb\leq1}1\right)+1\\
&\leq\sum_{(u,v):\text{ }u\leq[\frac{1}{a}],\text{ }v\leq[\frac{1}{b}]}1\\
&\leq\frac{1}{ab}.
\end{align*}

So for the asymptotic behavior, we have $s\approx\frac{3}{\pi^2}\frac{1}{ab}$; for every excursion, we have an upper bound:
\begin{equation}
s\leq\frac{1}{ab}.    \label{35}
\end{equation}

The equality is achieved when $a=b=1$ and $s=1$.

\end{rrr}

Notice that every x-coordinate of an excursion can be decided by the x-coordinates of two endpoints. We have the following lemma:

\begin{llll}   \label{75}
For any $(a,b)\in(0,1]^2$, there exists a unique excursion $(a_i,b_i)_{i=0}^n$ such that $a_0=a, a_n=b$.

\end{llll}

\begin{proof}[proof]
Existence:

Let $\Lambda_0$ be the lattice generated by $(a,0)$ and $(b,\frac{1}{a})$, $b_0=b+[\frac{1-b}{a}]a$. Since $(b_0,\frac{1}{a})=(b,\frac{1}{a})+[\frac{1-b}{a}](a,0)$, so $(b_0,\frac{1}{a})\in\Lambda_0$ and we know that $(a,0)$ and $(b_0,\frac{1}{a})$ are also a set of basis. Since $a_0,b_0\in(0,1]$, $a_0+b_0>1$, we have $\Lambda_0=\Lambda_{a_0,b_0}$. Let $(a_i,b_i)=T^i(a,b)$ for $i\geq0$, by Remark 2.1 of \cite{me}, we know $a_i$ is the x-coordinate of the primitive point with the $i$-th smallest slope in $(0,1]\times\mathbb{R}^+$. Therefore if $(b,\frac{1}{a})$ has the s-th smallest slope, then $a_s=b$. Also, $a_i$ would be greater than $a$ and $b$ for $i\in[1,s-1]$ because $(a,0)$ and $(b,\frac{1}{a})$ are the basis of $\Lambda_{a_0,b_0}$. This means that we have an excursion $(a_i,b_i)_{i=0}^s$ such that $a_0=a, a_s=b$.\\

Uniqueness:

Assuming there exists an excursion $(a_i,b_i)_{i=0}^s$ such that $a_0=a, a_s=b$, then by (\ref{64}), $s-1$ is the number of coprime pairs $(u,v)$ such that $ua+vb\leq 1$ and $a_i=u_ia+v_ib$ for $i\in[1,s-1]$. Thus, $b_i=a_{i+1}$ for $i\in[0,N-1]$, $(a_s,b_s)=T(a_{s-1},b_{s-1})$. When $a,b$ are fixed, all the information about the excursion is fixed, which shows the uniqueness of the excursion.

\end{proof}

According to Lemma \ref{75}, we can define the moduli space of excursions.

\begin{ddd}    \label{63}    (Moduli space of excursions)

We call $\Xi:=(0,1]^2$ the moduli space of excursions since there is a one-to-one correspondence between every point of $(0,1]^2$ and every excursion.

\end{ddd}

For any $(p,q)=1$ and $n\geq\max\{p,q\}$, $(\frac qn,\frac pn)$ is a point in the moduli space of excursions. We have the following theorem. Prior to stating it, let us clarify the notation used. $a_n\sim b_n$ means that $\displaystyle\lim_{n\rightarrow\infty}\frac{a_n}{b_n}=1$.

\begin{ttt}  \label{106}
For any piecewise continuous bounded function $f$ defined on the Farey triangle, $(p,q)=1$, we have
$$\lim_{n\rightarrow\infty}\frac{pq}{A_n}\sum_{(a,b)\in e\left(\frac qn,\frac pn\right)}
f(a,b)=\int_{\Omega}fdm,$$
where $e(x,y)$ is the excursion corresponding to $(x,y)$.

\end{ttt}

\begin{proof}[proof]
For any $(p,q)=1$, there exists $s,t\geq1$ such that $qt-ps=1$. Then for $n>\max\{p,q\}$, $\frac sq$ and $\frac tp$ are two fractions in $\mathcal{F}(n)$ while the denominator of any fraction between $\frac sq$ and $\frac tp$ is greater than $\max\{p,q\}$. Let all the denominators of the fractions in between are $q_0=q, q_1,\cdots, q_k=p$. By Lemma \ref{52}, we know that $\left(\frac{q_i}{n},\frac{q_{i+1}}{n}\right)_{i=0}^k$ is a section of BCZ orbit where $q_{k+1}$ is the denominator of the fraction after $\frac tp$. Since $q_i>\max\{p,q\}$ for $i\in[1,k-1]$, this orbit is an excursion which corresponds to $\left(\frac qn,\frac pn\right)$.

By Theorem 1.3 of \cite{At}, let $[\alpha,\beta]=\left[\frac sq, \frac tp\right]$, we have $\frac{1}{A_I(N)}\sum\limits_{i:\rho_i\in I}\delta_{T^i\left(\frac1Q,1\right)}=\frac{1}{\left|e\left(\frac qn,\frac pn\right)\right|}\sum\limits_{(a,b)\in e\left(\frac qn,\frac pn\right)}\delta_{(a,b)}$ weakly converge to Lebesgue uniform measure $dm=2dadb$ where $|e(x,y)|$ is the length of the excursion. Thus
$$\lim_{n\rightarrow\infty}\frac{1}{|e\left(\frac qn,\frac pn\right)|}\sum_{(a,b)\in e\left(\frac qn,\frac pn\right)}
f(a,b)=\int_{\Omega}fdm.$$
By Lemma \ref{65}, we have
$$\left|e\left(\frac qn,\frac pn\right)\right|\sim \frac{3}{\pi^2}\frac{n^2}{pq}\sim \frac{A_n}{pq},$$
which completes the proof.
\end{proof}

Theorem \ref{106} can be a useful tool when studying a section of a periodic orbit. Just like in the proof of Theorem \ref{86}, when we divide the orbit into several shorter sections, we could use Theorem \ref{106} in the further study of related problems.
One simple application is Lemma \ref{107}.
If $\frac pq$ is irrational, we are curious if there is a similar result. We conjecture that Theorem \ref{106} still holds.

Another property of the excursion is that the reverse of an excursion is still an excursion. But first, we need to establish that the reverse of a BCZ orbit is still a BCZ orbit.

\begin{llll}   \label{82}   (reverse of BCZ orbit)

For a BCZ orbit $T^i(a_0,b_0)=(a_i,b_i)$, $i\in\mathbb{Z}$, then $T^i(b_{s-1},a_{s-1})=(b_{s-1-i},a_{s-1-i})$, $i\in\mathbb{Z}$ is also a BCZ orbit.

\end{llll}

\begin{proof}[proof]

By the definition of the BCZ map,
$$a_{m+2}=b_{m+1}=k(a_m,b_m)b_m-a_m=k(a_m,b_m)a_{m+1}-a_m,$$
$$k(a_m,b_m)=\frac{a_m+a_{m+2}}{a_{m+1}},$$
$$k^T(a_m,b_m)=k(b_m,a_m)=\left[\frac{1+b_m}{a_m}\right]=\left[\frac{1+a_{m+1}}{a_m}\right].$$

Since
$$\frac{a_{m-1}+a_{m+1}}{a_m}\leq \frac{1+a_{m+1}}{a_m}<\frac{a_{m-1}+a_m+a_{m+1}}{a_m}=\frac{a_{m-1}+a_{m+1}}{a_m}+1,$$
we have
\begin{equation}
k^T(a_m,b_m)=\left[\frac{1+a_{m+1}}{a_m}\right]=\frac{a_{m-1}+a_{m+1}}{a_m}=k(a_{m-1},b_{m-1}).   \label{9}
\end{equation}
Thus,
$$T(b_m,a_m)=\left(a_m,k(b_m,a_m)a_m-b_m\right)=(a_m,a_{m-1}+a_{m+1}-a_{m+1})=(b_{m-1},a_{m-1}),$$
for $m\in\mathbb{Z}$.
Therefore,
$$T^i(b_{s-1},a_{s-1})=(b_{s-1-i},a_{s-1-i}),$$
for $i\in\mathbb{Z}$, which means that $T^i(b_{s-1},a_{s-1})=(b_{s-1-i},a_{s-1-i})$, $i\in\mathbb{Z}$ is also a BCZ orbit.

\end{proof}

\begin{ccc}   \label{83}
For an excursion $(a_m,b_m)_{m=0}^s$, $(b_{s-1-m},a_{s-1-m})_{m=0}^{s}$ is also an excursion.
\end{ccc}

\begin{proof}[proof]
By Lemma \ref{82}, we know that $(b_{s-1-m},a_{s-1-m})_{m=0}^{s}$ is a BCZ orbit, since
$$b_m=a_{m+1}>\max\{a_0,a_s\}=\max\{b_{-1},b_{s-1}\},$$
for $m\in[0,s-2]$, which means that $(b_{s-1-m},a_{s-1-m})_{m=0}^{s}$ is an excursion.
\end{proof}

In the moduli space of excursions, $(a_m,b_m)_{m=0}^s$ is corresponding to $(a_0,a_s)$ while $(b_{s-1-m},a_{s-1-m})_{m=0}^{s}$ is corresponding to $(b_{s-1},b_{-1})=(a_s,a_0)$.

\subsection{Reset of $\hat{k}-3$}

Now, we will demonstrate the estimation of the summation of $\hat{k}-3$ over the excursion, which will assist us in providing an estimation of the reset. The bound we provide will solely depend on the ratio of the x-coordinates of the two endpoints of the excursion, whereas the length of the excursion and the absolute value of the x-coordinates of the two endpoints will be irrelevant. If the x-coordinates of two endpoints are very small but closed to each other, then the excursion will be very long, according to Lemma \ref{65}. Nevertheless, the absolute value of the summation of $\hat{k}-3$ over the excursion will be very small.

\begin{ttt}    \label{66}

If $(a_m,b_m)_{m=0}^s$ is an excursion, we have

$$\sum\limits_{m=0}^{s-1}\left(\hat{k}(a_m,b_m)-3\right)\in\left(\frac{a_s}{a_0}+\frac{a_0}{a_s}-4,\frac{a_s}{a_0}+\frac{a_0}{a_s}-2\right).$$

\end{ttt}

\begin{proof}[proof]

By (\ref{9}), we have
$$k^T(a_m,b_m)=\left[\frac{1+a_{m+1}}{a_m}\right]=\frac{a_{m-1}+a_{m+1}}{a_m},$$
thus, by the definition of $h$, we know that
\begin{align}
\sum\limits_{m=0}^{s-1}\left(\hat{k}(a_m,b_m)-3\right)&=h(a_{-1},a_0,a_1,\cdots,a_{s+1})  \nonumber \\
&=h(a_{-2},a_{-1},a_0,a_1,\cdots,a_{s+2}) \nonumber \\
&-\left(\frac{\frac{a_{-2}+a_{0}}{a_{-1}}+\frac{a_{-1}+a_{1}}{a_{0}}}{2}-3\right)
-\left(\frac{\frac{a_{s-1}+a_{s+1}}{a_{s}}+\frac{a_{s}+a_{s+2}}{a_{s+1}}}{2}-3\right).    \label{36}
\end{align}

Next, we are going to calculate $h(a_{-2},a_{-1},a_0,a_1,\cdots,a_{s+2})$.

By example 1 of Remark \ref{62}, $(a_i)_{i=-2}^{s+2}$ is a generalized arithmetic sequence. Since $(a_i,b_i)_{i=0}^s$ is an excursion, if we pick a largest term $a_m$ among $(a_i)_{i=0}^{s}$, we know that $m\in[1,s-1]$. Because $a_0$ is smaller than every term of $(a_i)_{i=1}^{s-1}$, including $a_m$, by Lemma \ref{80}, we know that $a_m$ is a local maximum. So if we eliminate $a_m$ from $(a_i)_{i=-2}^{s+2}$, by Lemma \ref{59}, we know that the new sequence is still a generalized arithmetic sequence. Also, by Theorem \ref{60}, we know that the value of $h$ at the sequence will stay the same.

Then we can continue this process of eliminating largest terms between $a_0$ and $a_s$ from the new sequence until we eliminate all the terms between $a_0$ and $a_s$. After every elimination, the new sequence will still be a generalized arithmetic sequence and the value of $h$ at the new sequence will remain the same.

After eliminating all the terms between $a_0$ and $a_s$, we have a new sequence:
$$(a_{-2},a_{-1},a_0,a_s,a_{s+1},a_{s+2}),$$
so we have
\begin{equation}
h(a_{-2},a_{-1},a_0,a_1,\cdots,a_{s+2})=h(a_{-2},a_{-1},a_0,a_s,a_{s+1},a_{s+2}).   \label{37}
\end{equation}
Then by (\ref{36}), (\ref{37}), we have
\begin{align}
&h(a_{-1},a_0,a_1,\cdots,a_{s+1}) \nonumber \\ \nonumber
=&h(a_{-2},a_{-1},a_0,a_s,a_{s+1},a_{s+2})-\left(\frac{\frac{a_{-2}+a_{0}}{a_{-1}}+\frac{a_{-1}+a_{1}}{a_{0}}}{2}-3\right)\\ \nonumber
&-\left(\frac{\frac{a_{s-1}+a_{s+1}}{a_{s}}+\frac{a_{s}+a_{s+2}}{a_{s+1}}}{2}-3\right)\\
=&\frac{a_s}{a_0}+\frac{a_{-1}-a_1}{2a_0}+\frac{a_0}{a_s}+\frac{a_{s+1}-a_{s-1}}{2a_s}-3.  \label{91}
\end{align}

Since $-a_0<a_{-1}-1\leq a_{-1}-a_1\leq1-a_1<a_0$, $-a_s<a_{s+1}-1\leq a_{s+1}-a_{s-1}\leq1-a_{s-1}<a_s$, we have

$$\frac{a_s}{a_0}+\frac{a_0}{a_s}-4<h(a_{-1},a_0,a_1\cdots,a_s,a_{s+1})<\frac{a_s}{a_0}+\frac{a_0}{a_s}-2.$$

\end{proof}

\begin{rrr}   \label{92}
Another perspective to understanding this theorem is to transform $\hat{k}$.\\
\begin{align*}
 \hat{k}(a_i,b_i) - 3 &=
  \frac12\left(\frac{a_i}{a_{i+1}} + \frac{a_{i+1}}{a_i} \right) +
  \frac12\left(\frac{a_{i+2}}{a_{i+1}} + \frac{a_{i-1}}{a_i}\right) - 3\\
  &= \left(\frac{a_i}{a_{i+1}} + \frac{a_{i+1}}{a_i} - 3\right) +
  \frac12\left(\frac{a_{i+2}-a_i}{a_{i+1}} -
           \frac{a_{i+1}-a_{i-1}}{a_i}\right).
\end{align*}

For a local maximum $a_i=a_{i-1}+a_{i+1}$, we have
$$\left(\frac{a_i}{a_{i-1}} + \frac{a_{i-1}}{a_i} - 3\right)+\left(\frac{a_i}{a_{i+1}} + \frac{a_{i+1}}{a_i} - 3\right)=
\left(\frac{a_{i-1}}{a_{i+1}} + \frac{a_{i+1}}{a_{i-1}} - 3\right).$$

In order to calculate $\sum\limits_{m=0}^{s-1}\left(\hat{k}(a_m,b_m)-3\right)$, just like the proof of Theorem \ref{66}, we can eliminate all the terms between $a_0$ and $a_s$. Therefore, we have
\begin{align*}
\sum\limits_{m=0}^{s-1}\left(\hat{k}(a_m,b_m)-3\right)&=\sum\limits_{m=0}^{s-1}\left(\frac{a_i}{a_{i+1}} + \frac{a_{i+1}}{a_i} - 3\right)+\sum\limits_{m=0}^{s-1}\frac12\left(\frac{a_{i+2}-a_i}{a_{i+1}} -
           \frac{a_{i+1}-a_{i-1}}{a_i}\right)\\
&=\left(\frac{a_i}{a_{i+1}} + \frac{a_{i+1}}{a_i}-3\right)+\frac12\left(\frac{a_{s+1}-a_{s-1}}{a_{s}} -
           \frac{a_{1}-a_{-1}}{a_0}\right).
\end{align*}
Thus we get (\ref{91}).

\end{rrr}

\begin{ccc}    \label{85}
From Theorem \ref{66}, since $\frac{a_s}{a_0}+\frac{a_0}{a_s}\geq2$, we have
$$\frac{a_s}{a_0}+\frac{a_0}{a_s}>\sum\limits_{m=0}^{s-1}\left(\hat{k}(a_m,b_m)-3\right)>\frac{a_s}{a_0}+\frac{a_0}{a_s}-4\geq-2\geq-\left(\frac{a_s}{a_0}+\frac{a_0}{a_s}\right).$$
Thus,
\begin{equation}
\left|\sum\limits_{m=0}^{s-1}\left(\hat{k}(a_m,b_m)-3\right)\right|<\frac{a_s}{a_0}+\frac{a_0}{a_s}.     \label{39}
\end{equation}

\end{ccc}

We can use Theorem \ref{66} to give the estimation of the reset terms $\theta_m$ of (\ref{42}) for $m=r_i$.

\begin{llll}    \label{67}
For $i\geq1$, let $\theta'_i=\hat{k}(T^{i-1}(\frac{1}{n},1))-3$ and $\theta_i=\sum\limits_{j=1}^i\theta'_j$, then
$$\theta_{r_i}\in\left(i+\frac{1}{i}-4,i+\frac{1}{i}-2\right)$$
for $i\geq1$
\end{llll}
\begin{proof}[proof]
Since $\frac{p_{r_i}}{q_{r_i}}=\frac{1}{i}$, we know that $q_j>i$ for $1\leq j\leq r_i-1$. So $a_j=\frac{q_j}{n}>\frac{q_{r_i}}{n}=a_{r_i}\geq a_0$, therefore, $\left\{T^j\left(\frac{1}{n},1\right)\right\}_{j=0}^{r_i}$ is an excursion. Then by Theorem \ref{66}, we have
\begin{align*}\theta_{r_i}=\sum\limits_{m=0}^{r_i-1}\left(\hat{k}(a_m,b_m)-3\right)&\in\left(\frac{a_{r_i}}{a_0}+\frac{a_0}{a_{r_i}}-4,\frac{a_{r_i}}{a_0}+\frac{a_0}{a_{r_i}}-2\right)\\
&=\left(i+\frac{1}{i}-4,i+\frac{1}{i}-2\right).
\end{align*}

\end{proof}

\begin{rrr}
One thing worth mentioning is that the interval of estimation of $\sum\limits_{m=0}^{s-1}\left(\hat{k}(a_m,b_m)-3\right)$ in Theorem \ref{66} is only related to the ratio of $a_0$ and $a_s$. And in Lemma \ref{67}, the interval of estimation of $\theta_{r_i}$ is unrelated to $n$, unlike $|\iota_i|=O(n\log q_i)$, which we mentioned before in the Introduction.

\end{rrr}

The main purpose of estimating the reset terms is to help control $\sum\limits_{i=1}^{A_n}|\theta_i|$ of the main result (\ref{42}).\\

\subsection{Overall monotonicity of $\hat{k}-3$}

In \cite{me}, we prove the overall monotonicity of $k-3$. It states that the partial sum of $k-3$ over an excursion reaches its maximum at the beginning, whereas reaching its minimum near the end. In fact, this property also applies to $\hat{k}-3$.

\begin{ttt}   \label{84}
For an excursion $(a_m,b_m)_{m=0}^s$ where $s\geq4$, let $$\zeta_m=\sum_{i=0}^{m-1}\left(\hat{k}\left(T^i(a_0,b_0)\right)-3\right),$$ then we have
$$\zeta_{s-1}<\zeta_m<\zeta_1$$
for $m\in[2,s-2]$.

\end{ttt}

The proof of Theorem \ref{84} follows a similar argument as the proof of the overall monotonicity of $k-3$ in \cite{me}.

\begin{rrr}
For
$$\iota_{i+1}=\sum_{j=0}^i\left(R\left(T^j(\frac{1}{n},1)\right)-\frac{n^2}{A_n}\right),$$
which we mentioned in (\ref{48}), we attempt to control $\iota_i$ using $\iota_1$ and $\iota_{A_n-1}$. It seems that it can also be controlled and satisfies a similar inequality to the one in Theorem \ref{84}
$$\iota_{A_n-1}<\iota_m<\iota_1$$
for $m\in[2,A_n-2]$.

However, this is very hard to prove. But if we consider the approximated version of $R-\frac{n^2}{A_n}$ which is $\hat{k}-3$, we can find out that $(\iota_m)_{m=0}^{A_n}$ and $\left(\zeta_m\right)_{m=0}^{A_n}$ are very close but we can prove that 1-st term is the upper bound while the $A_n-1$-th term is the lower bound for the latter one.

\end{rrr}

\begin{rrr}    \label{90}
For $s\geq4$, from Theorem \ref{84}, we know that for any $m\in[2,s-2]$, $\zeta_{s-1}<\zeta_m<\zeta_1$.

Since $a_1+a_0>1$, $a_1>a_0$, so $a_1>\frac{1}{2}$, $a_0+a_2<a_1+1<3a_1$. Since $a_1\mid a_0+a_2$, we know that $\frac{a_0+a_2}{a_1}\leq2$. Therefore,
\begin{equation}
\zeta_1=\frac{\frac{a_{-1}+a_1}{a_0}+\frac{a_0+a_2}{a_1}}{2}-3\leq\frac{\frac{2}{a_0}+2}{2}-3=\frac{1}{a_0}-2.    \label{50}
\end{equation}
On the other side, for $\zeta_{s-1}$
$$\zeta_{s-1}=\zeta_s-\left(\hat{k}(a_{s-1},b_{s-1})-3\right)=\zeta_s-\zeta'_1.$$
By Corollary \ref{85}, we know that
$$|\zeta_s|<\frac{a_s}{a_0}+\frac{a_0}{a_s}<\frac{1}{a_0}+\frac{1}{a_s}.$$
By (\ref{50}), we can deduce that $\zeta'_1\leq\frac{1}{a_s}-2$, thus
\begin{equation}
\zeta_{s-1}>-\frac{1}{a_0}-\frac{2}{a_s}+2.     \label{51}
\end{equation}

By (\ref{50}), and (\ref{51}), we have
$$\zeta_m=O\left(\frac{1}{a_0}+\frac{1}{a_s}\right).$$

\end{rrr}

\section{Discretized analog of the RH}
\label{103}
In this section, we will provide the proof of Theorem \ref{55}, which is our main result. We define the energy function in Definition \ref{68}. Subsequently, we use induction on the energy function on the moduli space of excursions to establish a more generalized version of our main result (Theorem \ref{87}). Then, we utilize all the estimates on $\left(\frac{1}{n},\frac{1}{n}\right)$, which correspond to all the excursions that are also full periodic orbits to complete the proof of our main result.

First of all, given that $A_n=\frac{3}{\pi^2}n^2+O(n\log n)$, the main result Theorem \ref{55} is equivalent to Theorem \ref{86}.

\begin{ttt}    \label{86}
For $n\geq1$, $i\geq1$, let $\theta'_i=\hat{k}\left(T^{i-1}(\frac{1}{n},1)\right)-3$ and $\theta_i=\sum\limits_{j=1}^i\theta'_j$. Then, we have
$$\sum\limits_{i=1}^{A_n}|\theta_i|=O(n^{2+\epsilon}).$$
\end{ttt}

In Theorem \ref{86}, we observe that $\theta_i$ is derived from the period of a periodic orbit, specifically $\left(T^{i-1}\left(\frac{1}{n},1\right)\right)_{i=1}^{A_n}$. By Lemma \ref{52}, we know that it is $\left(\left(\frac{q_{i}}{n},\frac{q_{i+1}}{n}\right)\right)_{i=0}^{A_n-1}$. As we know, this section of a periodic orbit constitutes an excursion (explained in the examples following Definition \ref{74}). So, what about other excursions and functions? Could we define a similar expression to $\sum\limits_{i=1}^{A_n}|\theta_i|$? In other words, from the perspective of the moduli space of excursions, the formula in Theorem \ref{86} is only related to the point $(\frac{1}{n},\frac{1}{n})$ in the moduli space. Could it be related to all the points in the moduli space?

\begin{ddd}    \label{68}    (energy function)

For any $(a,b)\in\Xi=(0,1]^2$, by Lemma \ref{75}, there exists a unique BCZ orbit $T^i(a_0,b_0)=(a_i,b_i)$, $i\in[0,s]$, such that $a_0=a$, $a_s=b$ while from $(a_0,b_0)$ to $(a_s,b_s)$ is an excursion which means that $a_i>a,b$ for $i\in[1,s-1]$. Let $f$ be a function defined on the Farey triangle.
Then, for $i\in[1,s]$, let $\zeta'_i=f\left(T^{i-1}(a_0,b_0)\right)$, $\zeta_i=\sum\limits_{j=1}^i\zeta'_j$. We define the energy function of $f$ as
$$E(f;a,b):=\sum\limits_{i=1}^s|\zeta_i|$$
for $(a,b)\in\Xi$.

\end{ddd}

\begin{ttt}   \label{87}
Let
\begin{equation}
\Delta:=\left\{(a,b)|(a,b)\in\Xi, \frac{b}{a}\in\left[\frac{\sqrt{5}-1}{2},\frac{\sqrt{5}+1}{2}\right]\right\}   \label{7986969}
\end{equation}
be the golden ratio area of $\Xi$.
Then
$$E|_{\Delta}(\hat{k}-3;a,b)=O\left(\frac{1}{(ab)^{1+\epsilon}}\right).$$

In other words, for any $e>1$, there exists $C>0$ such that
$$E(\hat{k}-3;a,b)<C\frac{1}{(ab)^{e}}$$
for all $(a,b)\in\Delta$.

\end{ttt}

Before giving the proof of Theorem \ref{87}, we need to establish some lemmas. The notation will be the same as Definition \ref{68}. In the remainder of this section, we will simply use $E(a,b)$ instead of $E(\hat{k}-3;a,b)$ for brevity.

\begin{llll}   \label{88}
If $\max\{\frac{1}{a},\frac{1}{b}\}\leq d$ for $d\geq1$, then
\begin{equation}
E(a,b)\leq2d^5.   \label{28}
\end{equation}

\end{llll}

\begin{llll}   \label{89}
If $\max\{\frac{1}{a},\frac{1}{b}\}>\sqrt{5}+2$ and $(a,b)\in\Delta$, we have
\begin{equation}
2a+b<1,   \label{17}
\end{equation}
$$a+2b<1.$$

\end{llll}

Next, we define the sub-excursion, which means an excursion within another excursion.

\begin{ddd}      \label{69}    (sub-excursion)

We call an excursion $(c_j,d_j)_{j=0}^t$ a sub-excursion of an excursion $(a_i,b_i)_{i=0}^s$ if there exists $i_1,i_2\in[0,s]$ such that $a_{i_1}=c_0$, $a_{i_2}=c_t$ and $a_i<a_{i_1},a_{i_2}$ for $i\in(i_1,i_2)$. Then by the uniqueness of the excursion given two endpoints, we know that $i_2=i_1+t$ and $c_j=a_{i_1+j}$.

\end{ddd}

\begin{llll}   \label{70}
For an excursion $(a_i,b_i)_{i=0}^s$ where $a_0=a$,$a_s=b$ and a sub-excursion of its: $(c_j,d_j)_{j=0}^t$ where $a_{i_1}=c_0=c$, $a_{i_2}=c_t=d$, we have
$$\sum\limits_{m=i_1+1}^{i_2}|\zeta_m|\leq(i_2-i_1)|\zeta_{i_1}|+E(c,d).  $$

\end{llll}

\begin{proof}[proof](of Theorem \ref{87})

To prove the theorem, we need to show any $e>1$, there exists a constant $C>0$ such that
$$E(a,b)<\frac{C}{(ab)^e}$$
for all $(a,b)\in\Delta$.

Let $D=\left(\frac{\sqrt{5}-1}{6}\right)^{e-1}<1$. This choice of $D$ ensures that a certain inequality holds later in the proof. We then pick $d_1>0$ such that for any $d> d_1$,
\begin{equation}
2d+5d^2+2d^2\log d<\left(\frac{\sqrt{5}-1}{2}\right)^{e}(1-D)(d-1)^{2e}.           \label{26}
\end{equation}
This condition is crucial for bounding $E(a,b)$ in the following proof.

Let $d_2=\sqrt{5}+2$, $d_0=\max\{d_1,d_2\}$.

When $\max\{\frac{1}{a},\frac{1}{b}\}\leq d_0$, by Lemma \ref{88}, we have
\begin{equation}
E(a,b)\leq2d_0^5.     \label{38}
\end{equation}

Let $C=2d_0^5\geq2(\sqrt{5}+2)^5>1$.

Next, we use induction to prove that $E(a,b)<\frac{C}{(ab)^e}$ for $(a,b)\in\Delta$ .

(1) For $(a,b)\in\Delta$, when $\max\{\frac{1}{a},\frac{1}{b}\}\in[1,d_0]$, we have
$$E(a,b)< C\leq \frac{C}{(ab)^e}.$$

(2) Now, assume for $(a,b)\in\Delta$, when $\max\{\frac{1}{a},\frac{1}{b}\}\in[1,d-1]$, we have
$$E(a,b)< \frac{C}{(ab)^e},$$
where $d=d_0+d'$, $d'\in\mathbb{Z}^+$.

We now consider the case when $(a,b)\in\Delta$ and $\max\{\frac{1}{a},\frac{1}{b}\}\in(d-1,d]$.

By Lemma \ref{64}, we know that $\left\{a_i|i\in[1,s-1]\right\}=\{ua_0+va_s|u,v\in\mathbb{Z}^+,(u,v)=1,ua_0+va_s\leq1\}$. For $i\in[1,s-1]$, we let $a_i=u_ia_0+v_ia_s$, with $u_i,v_i\in\mathbb{Z}^+$ and $(u_i,v_i)=1$. Moreover, we know that $\frac{v_i}{u_i}<\frac{v_j}{u_j}$ for $1\leq i<j\leq s-1$.

Let $t_i\in[1,s-1]$ be the number such that $(u_{t_i},v_{t_i})=(i,1)$ for $1\leq i\leq[\frac{1-b}{a}]$.

Since $\max\{\frac{1}{a},\frac{1}{b}\}>d-1\geq d_0\geq d_2$, by (\ref{17}) of Lemma \ref{89}, we know that $[\frac{1-b}{a}]\geq2$. Let $c=[\frac{1-b}{a}]$.

For any $i\in[1,s-1]$, $u_i\leq[\frac{1-b}{a}]=c=u_{t_c}$, which means that $\frac{v_{t_c}}{u_{t_c}}=\frac{1}{u_{t_c}}\leq\frac{v_i}{u_i}$. Thus, we know that $t_c=1$.

For any $j\in(0,t_i)$ where $i\in\left[1,[\frac{1-b}{a}]\right]$, $\frac{v_j}{u_j}<\frac{1}{i}$, so $v_j\geq1$, $u_j>i$, $a_j>a_{t_i}$. So it is an excursion from $(a_0,b_0)$ to $(a_{t_i},b_{t_i})$, by Corollary \ref{85}, we have
$$|\zeta_{t_i}|<\frac{a_{t_i}}{a_0}+\frac{a_0}{a_{t_i}},$$
so
\begin{equation}
|\zeta_{t_i}|<\frac{a_{t_i}}{a_0}+\frac{a_0}{a_{t_i}}=\frac{ia+b}{a}+\frac{a}{ia+b}<i+\frac{\sqrt{5}+3}{2}.   \label{3}
\end{equation}

Besides, by (\ref{35}), we know that
\begin{equation}
t_i<\frac{1}{a_0a_{t_i}}=\frac{1}{a(ia+b)}.     \label{23}
\end{equation}

By (\ref{23}), when $i=1$, we have
\begin{equation}
t_1<\frac{1}{a(a+b)}                      \label{24}
\end{equation}

On the other side, for $i\in(t_1,s)$, $a_i=u_ia+v_ib>a+b=a_{t_1}>b=a_s$, so it is also an excursion from $(a_{t_1},b_{t_1})$ to $(a_s,b_s)$ which means that
\begin{equation}
s-t_1<\frac{1}{a_sa_{t_1}}=\frac{1}{b(a+b)}.                      \label{25}
\end{equation}

For any $j\in(t_{i+1},t_i)$ where $i\in[1,c)$, $\frac{1}{i+1}<\frac{v_j}{u_j}<\frac{1}{i}$, so $v_j>1$, $u_j>i+1$, $a_j>a_{t_{i+1}},a_{t_i}$. So it is an excursion from $(a_{t_{i+1}},b_{t_{i+1}})$ to $(a_{t_i},b_{t_i})$, by Definition \ref{68}, we know that
$$E(a_{t_{i+1}},a_{t_i})=\sum\limits_{m=t_{i+1}}^{t_i-1}\left|\sum\limits_{j=t_{i+1}}^{m}(\hat{k}(a_j,b_j)-3)\right|,$$
since
$$\frac{b}{a}\in\left[\frac{\sqrt{5}-1}{2},\frac{\sqrt{5}+1}{2}\right],$$
we have
$$\frac{ia+b}{(i+1)a+b}\in\left[\frac{\sqrt{5}-1}{2},\frac{\sqrt{5}+1}{2}\right]$$
for $i\geq1$.

Since
$$\max\left\{\frac{1}{(i+1)a+b},\frac{1}{ia+b}\right\}\leq\frac{1}{a+b}<\frac{d}{2}<d-1,$$
by our hypothesis, we have
\begin{equation}
E(a_{t_{i+1}},a_{t_i})<\frac{C}{(((i+1)a+b)(ia+b))^e}.     \label{4}
\end{equation}

Since it is an excursion from $(a_{t_{i+1}},b_{t_{i+1}})$ to $(a_{t_i},b_{t_i})$, by Lemma \ref{70}, we obtain
\begin{align}
\sum\limits_{m=t_{i+1}+1}^{t_i}|\zeta_m|\leq(t_i-t_{i+1})|\zeta_{t_{i+1}}|+E(a_{t_{i+1}},a_{t_i}).      \label{5}
\end{align}

Using (\ref{5}), we derive
\begin{align}
\sum\limits_{m=t_c+1}^{t_1}|\zeta_m|=&\sum_{i=1}^{c-1}\sum\limits_{m=t_{i+1}+1}^{t_i}|\zeta_m| \nonumber \\
\leq&\sum_{i=1}^{c-1}\left[(t_i-t_{i+1})|\zeta_{t_{i+1}}|+E(a_{t_{i+1}},a_{t_i})\right].    \label{6}
\end{align}

Furthermore, by (\ref{35}), we know that
\begin{equation}
t_i-t_{i+1}<\frac{1}{a_{t_i}a_{t_{i+1}}}=\frac{1}{(ia+b)((i+1)a+b)}.       \label{22}
\end{equation}

For $\sum\limits_{i=1}^{c-1}(t_i-t_{i+1})|\zeta_{t_{i+1}}|$, by (\ref{3}),(\ref{24}), and (\ref{22}), we have
\begin{align}
\sum_{i=1}^{c-1}(t_i-t_{i+1})|\zeta_{t_{i+1}}|&<\sum_{i=1}^{c-1}(t_i-t_{i+1})\left(i+1+\frac{\sqrt{5}+3}{2}\right) \nonumber \\ \nonumber
&=\sum_{i=1}^{c-1}(t_i-t_{i+1})i+(t_1-t_c)\left(\frac{\sqrt{5}+5}{2}\right)\\ \nonumber
&<\sum_{i=1}^{c-1}\frac{1}{(ia+b)((i+1)a+b)}i+\frac{1}{a(a+b)}\left(\frac{\sqrt{5}+5}{2}\right)\\ \nonumber
&<\sum_{i=1}^{c-1}\frac{1}{(i+1)a^2}+\frac{4}{a(a+b)}\\
&<\frac{\log c}{a^2}+\frac{4}{a(a+b)}.     \label{19}
\end{align}

For $\sum\limits_{i=1}^{c-1}E(a_{t_{i+1}},a_{t_i})$, since for $i\geq1$,
$$\frac{(ab)^{e-1}}{\left(\left((i+1)a+b\right)(ia+b)\right)^{e-1}}<\left(\frac{ab}{(2a+b)(a+b)}\right)^{e-1},$$
$$=\left(\frac{1}{\frac{2a}{b}+3+\frac{b}{a}}\right)^{e-1}\leq\left(\frac{1}{3+\frac{3(\sqrt{5}-1)}{2}}\right)^{e-1}=D.$$

Thus by (\ref{4}), we have
\begin{align}
\sum\limits_{i=1}^{c-1}E(a_{t_{i+1}},a_{t_i})&<\sum\limits_{i=1}^{c-1}\frac{C}{(((i+1)a+b)(ia+b))^e} \nonumber \\ \nonumber
&=\frac{C}{(ab)^{e-1}}\sum\limits_{i=1}^{c-1}\frac{(ab)^{e-1}}{(((i+1)a+b)(ia+b))^e}\\ \nonumber
&<\frac{C}{(ab)^{e-1}}\sum\limits_{i=1}^{c-1}\frac{D}{(((i+1)a+b)(ia+b))}\\ \nonumber
&=\frac{CD}{(ab)^{e-1}a}\sum\limits_{i=1}^{c-1}\left(\frac{1}{ia+b}-\frac{1}{((i+1)a+b)}\right)\\ \nonumber
&=\frac{CD}{(ab)^{e-1}a}\left(\frac{1}{a+b}-\frac{1}{ca+b}\right)\\ \nonumber
&=\frac{CD(c-1)}{(ab)^{e-1}(a+b)(ca+b)}\\
&<\frac{CD}{(ab)^{e-1}(a+b)a}.     \label{21}
\end{align}

Then by (\ref{6}),(\ref{19}), and (\ref{21}), we obtain
\begin{equation}
\sum\limits_{m=t_c+1}^{t_1}|\zeta_m|<\frac{\log c}{a^2}+\frac{4}{a(a+b)}+\frac{CD}{(ab)^{e-1}(a+b)a}.   \label{7}
\end{equation}

By Corollary \ref{85}, we have
\begin{equation}
|\zeta_1|=\left|\hat{k}(a_0,b_0)-3\right|<\frac{a_0}{b_0}+\frac{b_0}{a_0}<\frac{2b_0}{a_0}\leq\frac{2}{a}.  \label{20}
\end{equation}

Since $t_c=1$ and $c=\left[\frac{1-b}{a}\right]\leq\frac{1-b}{a}<\frac{1}{a}$, using (\ref{7}) and (\ref{20}), we obtain
\begin{align}
\sum\limits_{m=1}^{t_1}|\zeta_m|=&\sum\limits_{m=1}^{t_c}|\zeta_m|+\sum\limits_{m=t_c+1}^{t_1}|\zeta_m| \nonumber \\ \nonumber
<&\frac{2}{a}+\frac{\log c}{a^2}+\frac{4}{a(a+b)}+\frac{CD}{(ab)^{e-1}(a+b)a}\\ \nonumber
<&d+\frac{\log\frac{1}{a}}{a^2}+\frac{4}{a(a+b)}+\frac{CD}{(ab)^{e-1}(a+b)a}\\
\leq&d+d^2\log d+2d^2+\frac{CD}{(ab)^{e-1}(a+b)a}.     \label{10}
\end{align}

For $|\zeta_m|$ where $m\in[t_1+1,s-1]$, we have
$$|\zeta_m|=\left|\zeta_s-\sum\limits_{j=m}^{s-1}\left(\hat{k}(a_j,b_j)-3\right)\right|\leq|\zeta_s|+\left|\sum\limits_{j=m}^{s-1}\left(\hat{k}(a_j,b_j)-3\right)\right|.$$

Since it is an excursion from $(a_0,b_0)$ to $(a_s,b_s)$, by Theorem \ref{66}, we have
$$\zeta_s\in\left(\frac{a}{b}+\frac{b}{a}-4,\frac{a}{b}+\frac{b}{a}-2\right),$$
Given that $\frac{b}{a}\in\left[\frac{\sqrt{5}-1}{2},\frac{\sqrt{5}+1}{2}\right]$, it follows that $\frac{a}{b}+\frac{b}{a}\in\left[2,\sqrt{5}\right]$, thus
\begin{equation}
\zeta_s\in\left(-2,\sqrt{5}-2\right).        \label{11}
\end{equation}

By Lemma \ref{82}, we have
$$T^j(b_{s-1},a_{s-1})=(b_{s-1-j},a_{s-1-j})$$
for $j\in\mathbb{Z}$.

Moreover, since $\hat{k}(a_m,b_m)=\hat{k}(b_m,a_m)$, we can get
\begin{align}
\sum\limits_{m=t_1+1}^{s-1}\left|\sum\limits_{j=m}^{s-1}\left(\hat{k}(a_j,b_j)-3\right)\right|&=\sum\limits_{m=t_1+1}^{s-1}\left|\sum\limits_{j=m}^{s-1}\left(\hat{k}(b_j,a_j)-3\right)\right| \nonumber \\ \nonumber
&=\sum\limits_{m=0}^{s-t_1-2}\left|\sum\limits_{j=0}^{m}\left(\hat{k}\left(T^j(b_{s-1},a_{s-1})\right)-3\right)\right|\\
&\leq\sum\limits_{m=0}^{s-t_1-1}\left|\sum\limits_{j=0}^{m}\left(\hat{k}\left(T^j(b_{s-1},a_{s-1})\right)-3\right)\right|.    \label{12}
\end{align}

As it is an excursion from $(a_{t_1},b_{t_1})$ to $(a_s,b_s)$, by Corollary \ref{83}, it is also an excursion from $(b_{s-1},a_{s-1})$ to $(b_{t_1-1},a_{t_1-1})$. Given that $b_{s-1}=b$ and $b_{t_1-1}=a+b$, we have
\begin{equation}
E(b,a+b)=\sum\limits_{m=0}^{s-t_1-1}\left|\sum\limits_{j=0}^{m}\left(\hat{k}\left(T^j(b_{s-1},a_{s-1})\right)-3\right)\right|.    \label{13}
\end{equation}

Since the orbit from $(a_0,b_0)$ to $(a_{t_1},b_{t_1})$ is an excursion, with $a_0=a$ and $a_{t_i}=a+b$, we know that $E(a,a+b)=\sum\limits_{m=1}^{t_1}|\zeta_m|.$
Therefore by (\ref{10}), we obtain
\begin{equation}
E(a,a+b)<d+d^2\log d+2d^2+\frac{CD}{(ab)^{e-1}(a+b)a}.      \label{33}
\end{equation}

Note that (\ref{33}) is only valid under the assumption that $\max\{\frac{1}{a},\frac{1}{b}\}\in[1,d-1]$, which means that we can exchange $a$ and $b$ in (\ref{33}), it will still hold. Thus, we have
\begin{equation}
E(b,a+b)<d+d^2\log d+2d^2+\frac{CD}{(ab)^{e-1}(a+b)b}.       \label{14}
\end{equation}

Therefore combining (\ref{11}),(\ref{12}),(\ref{13}), and (\ref{14}), we have
\begin{align}
\sum\limits_{m=t_1+1}^s|\zeta_m|<&(s-t_1)|\zeta_s|+\sum\limits_{m=t_1+1}^{s-1}\left|\sum\limits_{j=m}^{s-1}\left(\hat{k}(a_j,b_j)-3\right)\right| \nonumber \\ \nonumber
<&2(s-t_1)+\sum\limits_{m=0}^{s-t_1-1}\left|\sum\limits_{j=0}^{m}\left(\hat{k}\left(T^j(b_{s-1},a_{s-1})\right)-3\right)\right|\\ \nonumber
=&2(s-t_1)+E(b,a+b)\\
<&2(s-t_1)+d+d^2\log d+2d^2+\frac{CD}{(ab)^{e-1}(a+b)b}.   \label{15}
\end{align}

By (\ref{25}), we know that $2(s-t_1)<\frac{2}{(a+b)b}\leq d^2$. Thus, by (\ref{10}) and (\ref{15}), we obtain
\begin{align}
E(a,b)=&\sum\limits_{i=1}^s|\zeta_i| \nonumber \\ \nonumber
<&d+d^2\log d+2d^2+\frac{CD}{(ab)^{e-1}(a+b)a}+2(s-t_1)\\ \nonumber
&+d+d^2\log d+2d^2+\frac{CD}{(ab)^{e-1}(a+b)b} \\
<&2d+5d^2+2d^2\log d+\frac{CD}{(ab)^e}.                       \label{27}
\end{align}

Since $d> d_0\geq d_1$, by (\ref{26}), we have
$$E(a,b)<2d+5d^2+2d^2\log d+\frac{CD}{(ab)^e}<\left(\frac{\sqrt{5}-1}{2}\right)^{e}(1-D)(d-1)^{2e}+\frac{CD}{(ab)^e}.$$

Since $\max\left\{\frac{1}{a},\frac{1}{b}\right\}>d-1$, it follows that $\min\left\{\frac{1}{a},\frac{1}{b}\right\}>\frac{(\sqrt{5}-1)(d-1)}{2}$. This means that
$$\frac{1}{ab}>\frac{(\sqrt{5}-1)(d-1)^2}{2}.$$

Furthermore, given that $C>1$, we can conclude that
$$E(a,b)<\frac{1-D}{(ab)^e}+\frac{CD}{(ab)^e}<\frac{C(1-D)}{(ab)^e}+\frac{CD}{(ab)^e}=\frac{C}{(ab)^e},$$
which completes the induction.\\

\end{proof}

\begin{proof}[proof] (of Theorem \ref{86})

Since $(\frac{1}{n},\frac{1}{n})\in\Delta$, by Theorem \ref{87}, let $a=b=\frac{1}{n}$, we obtain
$$E\left(\frac{1}{n},\frac{1}{n}\right)=O(n^{2+\epsilon}).$$

As we know, from $(\frac{1}{n},1)$ to $T^{A_n}(\frac{1}{n},1)=(\frac{1}{n},1)$ is an excursion, so by definition of energy function $E$, we have
$$E\left(\frac{1}{n},\frac{1}{n}\right)=\sum\limits_{i=1}^{A_n}|\theta_i|.$$
Therefore
$$\sum\limits_{i=1}^{A_n}|\theta_i|=O(n^{2+\epsilon}).$$\\

\end{proof}

\begin{proof}[proof](of Lemma \ref{88})

When $\max\{\frac{1}{a},\frac{1}{b}\}\leq d$,
$$k(a_{i-1},a_{i})=\left[\frac{1+a_{i-1}}{a_{i}}\right]\leq\frac{2}{a_{i}}\leq2d,$$
$$k(a_i,a_{i-1})=\left[\frac{1+a_i}{a_{i-1}}\right]\leq\frac{2}{a_{i-1}}\leq2d$$
for $i\in[1,s]$.

Thus,
$$\zeta'_i=\hat{k}\left(T^{i-1}(a_0,b_0)\right)-3=\hat{k}\left(T^{i-1}(a_{i-1},a_i)\right)-3,$$
$$=\frac{k(a_{i-1},a_{i})+k(a_i,a_{i-1})}{2}-3\in[-2,2d-3]$$
for $i\in[1,s]$.

Therefore,
$$|\zeta_i|=\left|\sum\limits_{j=1}^i\zeta'_j\right|\leq\sum\limits_{j=1}^i|\zeta'_j|<2id\leq2sd$$
for $i\in[1,s]$.

By (\ref{35}), we have $s<\frac{1}{ab}$, so
$$E(a,b)=\sum\limits_{i=1}^s|\zeta_i|<2s^2d<\frac{2d}{a^2b^2}\leq2d^5.$$

\end{proof}

\begin{proof}[proof](of Lemma \ref{89})

When $\max\{\frac{1}{a},\frac{1}{b}\}>\sqrt{5}+2$, we have
$$2a+b\leq\left(1+2\left(\frac{\sqrt{5}+1}{2}\right)\right)\min\{a,b\}<\frac{\sqrt{5}+2}{\sqrt{5}+2}=1,$$
$$a+2b\leq\left(1+2\left(\frac{\sqrt{5}+1}{2}\right)\right)\min\{a,b\}<\frac{\sqrt{5}+2}{\sqrt{5}+2}=1.$$

\end{proof}

\begin{proof}[proof](of Lemma \ref{70})

By definition, for $m\in[i_1+1,i_2]$, we have
$$\zeta_m=\zeta_{i_1}+\sum\limits_{j=i_1}^{m-1}\left(\hat{k}(a_j,b_j)-3\right).$$

From Definition \ref{69}, we know that
$$\sum\limits_{m=i_1+1}^{i_2}\left|\sum\limits_{j=i_1}^{m-1}\left(\hat{k}(a_j,b_j)-3\right)\right|=\sum\limits_{m=1}^{t}\left|\sum\limits_{j=1}^{m-1}\left(\hat{k}(c_j,d_j)-3\right)\right|
=E(c,d).$$

Therefore,
\begin{align*}
\sum\limits_{m=i_1+1}^{i_2}|\zeta_m|&\leq\sum\limits_{m=i_1+1}^{i_2}\left(|\zeta_{i_1}|+\left|\sum\limits_{j=i_1}^{m-1}\left(\hat{k}(a_j,b_j)-3\right)\right|\right)\\
&=(i_2-i_1)|\zeta_{i_1}|+E(c,d).
\end{align*}
\end{proof}

\section{Sufficient Conditions for General Case}
\label{100}

In \S4, we obtain two properties of the discretized approximation function $\hat{k}-3$, which demonstrate that we have good control over it. Given
$$\zeta_m=\sum_{i=0}^{m-1}\left(\hat{k}\left(T^i(a_0,b_0)\right)-3\right),$$
we establish the following:

1. Reset control: Corollary \ref{85} asserts that
$$\zeta_s=O\left(\frac{a_s}{a_0}+\frac{a_0}{a_s}\right)$$
holds for an excursion from $(a_0,b_0)$ to $(a_s,b_s)$.

2: Overall monotonicity: Theorem \ref{84} indicates that
$$\zeta_{s-1}<\zeta_m<\zeta_1$$
holds for an excursion from $(a_i,b_i)_{i=0}^s$ with $m\in[2,s-2]$.

Additionally, Remark \ref{90} provides
$$\zeta_m=O\left(\frac{1}{a_0}+\frac{1}{a_s}\right).$$

The reset control serves as the key tool in the proof of our main result. Theorem \ref{86} establishes that for $\hat{k}-3$, we have $\sum\limits_{i=1}^{A_n}|\theta_i|=O(n^{2+\epsilon})$. By (\ref{43}), it suffices to prove that for $R-\frac{n^2}{A_n}$, we have $\sum\limits_{i=1}^{A_n}|\iota_i|=O\left(n^{\frac{5}{2}+\epsilon}\right)$ to verify the RH. This prompts the question: if a function satisfies weaker conditions than $\hat{k}-3$, can we still obtain weaker control over the corresponding energy function? The following theorems address this inquiry.

\subsection{Results and examples}

\begin{ttt}   \label{56}
For a function $g$ defined on the Farey triangle, let $\zeta'_i=g\left(T^{i-1}(a_0,b_0)\right)$ and $\zeta_i=\sum\limits_{j=1}^i\zeta'_j$, if
there exists $\alpha\geq1$ and $C_1>0$ such that for any excursion from $(a_0,b_0)$ to $(a_s,b_s)$,
$$|\zeta_s|<C_1\left(\left(\frac{a_0}{a_s}\right)^\alpha+\left(\frac{a_s}{a_0}\right)^\alpha\right)$$
holds, then we have
\begin{equation}
E(g;a_0,a_s)=\sum\limits_{i=1}^{s}|\zeta_i|=O\left(\left(\frac{1}{a_0a_s}\right)^{\frac{\alpha+1}{2}+\epsilon}\right)  \label{93}
\end{equation}
for $(a_0,a_s)\in\Delta$.

\end{ttt}

According to Remark \ref{92}, when
$$g(a,b)=\frac{a}{b}+\frac{b}{a}-3,$$
the function satisfies the condition of Theorem \ref{56} because
\begin{equation}
\zeta_s=\frac{a_0}{a_s}+\frac{a_s}{a_0}-3, \label{94}
\end{equation}
therefore, (\ref{93}) applies to $g(a,b)=\frac{a}{b}+\frac{b}{a}-3$.

Furthermore, for any $\lambda\in\mathbb{R}$, let
\begin{align}
g_{\lambda}(a,b)&=\lambda\left(\frac{a_1}{a_0}+\frac{a_0}{a_1}\right)+(1-\lambda)\left(\frac{a_{-1}}{a_0}+\frac{a_2}{a_1}\right)-3 \nonumber \\ \nonumber
&=\frac{\lambda a_1+(1-\lambda)a_{-1}}{a_0}+\frac{\lambda a_0+(1-\lambda)a_2}{a_1}-3\\
&=(2\lambda-1)\left(\frac{a_1}{a_0}+\frac{a_0}{a_1}-3\right)+(2-2\lambda)\left(\frac{\frac{a_{-1}+a_1}{a_0}+\frac{a_0+a_2}{a_1}}{2}-3\right).  \label{98}
\end{align}

By Theorem \ref{66}, (\ref{94}), and (\ref{98}), we know that
$$\zeta_s\in\left((2\lambda-1)(\frac{a_s}{a_0}+\frac{a_0}{a_s})-3-|2-2\lambda|,(2\lambda-1)(\frac{a_s}{a_0}+\frac{a_0}{a_s})-3+|2-2\lambda|\right),$$
therefore, $g_{\lambda}(a,b)$ also satisfies the condition of Theorem \ref{56}, which means that (\ref{93}) applies to $g_{\lambda}(a,b)$ as well.

In fact, $\forall \lambda_1,\lambda_2$, $\lambda_1\left(\frac{a_1}{a_0}+\frac{a_0}{a_1}-3\right)+\lambda_2\left(\frac{\frac{a_{-1}+a_1}{a_0}+\frac{a_0+a_2}{a_1}}{2}-3\right)$ satisfies the condition of Theorem \ref{56}.

More generally, the set of all functions that satisfy the conditions of Theorem \ref{56} forms a linear function space.

\begin{ttt}      \label{57}
For a function $g$ defined on the Farey triangle, let $\zeta'_i=g(T^{i-1}(a_0,b_0))$ and $\zeta_i=\sum\limits_{j=1}^i\zeta'_j$. If the following conditions hold:

(1): There exists $\alpha>1$ and $C_1>0$ such that for any excursion from $(a_0,b_0)$ to $(a_s,b_s)$, we have
$$|\zeta_s|<C_1\left(\left(\frac{a_0}{a_s}\right)^\alpha+\left(\frac{a_s}{a_0}\right)^\alpha\right).$$

(2): There exists $0<\beta<\alpha, C_2>0$, and $\gamma>0$ such that for any excursion from $(a_0,b_0)$ to $(a_s,b_s)$ where $(a_0,a_s)\in\Delta$,
$$|\zeta_i|<C_2\left(\left(\frac{1}{a_0}\right)^\beta+\left(\frac{1}{a_s}\right)^\beta\right)$$
holds for any $i\in\left[1,\gamma\left(\frac{1}{a_0}\right)^{2-\frac{\beta}{\alpha}}\right)$.

Then we have
$$E(g;a_0,a_s)=\sum\limits_{i=1}^{s}|\zeta_i|=O\left(\left(\frac{1}{a_0a_s}\right)^{1+\frac{\beta}{2}(1-\frac{1}{\alpha})+\epsilon}\right)$$
for $(a_0,a_s)\in\Delta$.

\end{ttt}

\subsection{Proof}

\begin{proof}[proof] (of Theorem \ref{56})

The proof is highly similar to the proof of Theorem \ref{87}. We will only highlight the differences in this proof and omit the common parts. In the remainder of this section, we will use $E(a,b)$ instead of $E(g;a,b)$ for conciseness.

We need to prove that for any $e>\frac{1+\alpha}{2}$,
$$E(a,b)<\frac{C}{(ab)^e}$$
for $(a,b)\in\Delta$ .

Instead of (\ref{26}), we pick $d_1>0$ such that for any $d> d_1$,
$$2C_1d^2+4C_1d+8C_1d^2\log d<\left(\frac{\sqrt{5}-1}{2}\right)^{e}(1-D)(d-1)^{2e}$$
when $\alpha=1$;
$$2^{\alpha}C_1d^2+4C_1d^\alpha+\frac{8C_1(d+1)^{\alpha-1}d^2}{\alpha-1}<\left(\frac{\sqrt{5}-1}{2}\right)^{e}(1-D)(d-1)^{2e}$$
when $\alpha>1$.

Let $d_0=\max\{d_1,d_2\}$.

When $\max\{\frac{1}{a},\frac{1}{b}\}\leq d_0$, we have
\begin{align*}\zeta'_i=g(a_{k-1},a_k)&<C_1\left(\left(\frac{a_{k-1}}{a_k}\right)^\alpha+\left(\frac{a_k}{a_{k-1}}\right)^\alpha\right)\\
&\leq C_1\left(\left(\frac{1}{a_k}\right)^\alpha+\left(\frac{1}{a_{k-1}}\right)^\alpha\right)
\leq2C_1d_0^\alpha\end{align*}
for $i\in[1,s]$.

Instead of (\ref{38}), we have
$$E(a,b)=\sum\limits_{i=1}^s|\zeta_i|\leq\sum\limits_{i=1}^s\sum\limits_{j=1}^i|\zeta'_j|<\sum\limits_{i=1}^s2C_1d_0^\alpha i$$
$$<2C_1s^2d_0^\alpha<\frac{2C_1d_0^\alpha}{a^2b^2}\leq2C_1d_0^{4+\alpha}.$$

Let $C=\max\{2C_1d_0^{4+\alpha},1\}$.\\

(1) For $(a,b)\in\Delta$, when $\max\{\frac{1}{a},\frac{1}{b}\}\in[1,d_0]$, we have
$$E(a,b)< C\leq \frac{C}{(ab)^e}.$$\\

(2) Now, assume for $(a,b)\in\Delta$, when $\max\{\frac{1}{a},\frac{1}{b}\}\in[1,d-1]$, we have
$$E(a,b)< \frac{C}{(ab)^e},$$
where $d=d_0+d'$, $d'\in\mathbb{Z}^+$.

We now consider the case when $(a,b)\in\Delta$, $\max\{\frac{1}{a},\frac{1}{b}\}\in(d-1,d]$.

Instead of (\ref{3}), we have
\begin{align*}|\zeta_{t_i}|<C_1\left(\left(\frac{a_{t_i}}{a_0}\right)^\alpha+\left(\frac{a_0}{a_{t_i}}\right)^\alpha\right)&=
C_1\left(\left(\frac{ia+b}{a}\right)^\alpha+\left(\frac{a}{ia+b}\right)^\alpha\right)\\
&<2C_1(i+1)^\alpha.\end{align*}

Instead of (\ref{19}), we have
\begin{align}
\sum_{i=1}^{c-1}(t_i-t_{i+1})|\zeta_{t_{i+1}}|&<\sum_{i=1}^{c-1}(t_i-t_{i+1})2C_1(i+1)^\alpha \nonumber  \\  \nonumber
&<\sum_{i=1}^{c-1}\frac{1}{(ia+b)((i+1)a+b)}2C_1(i+1)^\alpha\\ \nonumber
&<\sum_{i=1}^{c-1}\frac{1}{i(i+1)a^2}2C_1(i+1)^\alpha\\ \nonumber
&\leq\sum_{i=1}^{c-1}\frac{2}{(i+1)^2a^2}2C_1(i+1)^\alpha\\
&\leq\sum_{i=1}^{c-1}\frac{4C_1}{(i+1)^{2-\alpha}a^2}.       \label{29}
\end{align}

When $\alpha=1$,
$$\sum_{i=1}^{c-1}(t_i-t_{i+1})|\zeta_{t_{i+1}}|<\frac{4C_1\log c}{a^2};$$
when $\alpha>1$,
$$\sum_{i=1}^{c-1}(t_i-t_{i+1})|\zeta_{t_{i+1}}|<\frac{4C_1(c+1)^{\alpha-1}}{(\alpha-1)a^2}.$$

Instead of (\ref{7}), we have
$$\sum\limits_{m=t_c+1}^{t_1}|\zeta_m|<\frac{4C_1\log c}{a^2}+\frac{CD}{(ab)^{e-1}(a+b)a},$$
when $\alpha=1$;
$$\sum\limits_{m=t_c+1}^{t_1}|\zeta_m|<\frac{4C_1(c+1)^{\alpha-1}}{(\alpha-1)a^2}+\frac{CD}{(ab)^{e-1}(a+b)a},$$
when $\alpha>1$.

Instead of (\ref{20}), we have
$$|\zeta_1|<C_1\left(\left(\frac{a_{0}}{a_1}\right)^\alpha+\left(\frac{a_1}{a_{0}}\right)^\alpha\right)<2C_1\left(\frac{a_1}{a_0}\right)^\alpha\leq2C_1\left(\frac{1}{a}\right)^\alpha.$$

Instead of (\ref{10}), we have
$$\sum\limits_{m=1}^{t_1}|\zeta_m|<2C_1d+4C_1d^2\log d+\frac{CD}{(ab)^{e-1}(a+b)a},$$
when $\alpha=1$;
$$\sum\limits_{m=1}^{t_1}|\zeta_m|<2C_1d^\alpha+\frac{4C_1(d+1)^{\alpha-1}d^2}{\alpha-1}+\frac{CD}{(ab)^{e-1}(a+b)a},$$
when $\alpha>1$.

Instead of (\ref{11}), we have
$$|\zeta_s|<C_1\left(\left(\frac{a_{0}}{a_s}\right)^\alpha+\left(\frac{a_s}{a_{0}}\right)^\alpha\right)\leq2C_1\left(\frac{\sqrt{5}+1}{2}\right)^\alpha<2^{\alpha+1}C_1.$$

Instead of (\ref{15}) and (\ref{27}), we have
$$\sum\limits_{m=t_1+1}^s|\zeta_m|<4C_1(s-t_1)+2C_1d+4C_1d^2\log d+\frac{CD}{(ab)^{e-1}(a+b)b},$$
\begin{align*}
E(a,b)&<2C_1d^2+4C_1d+8C_1d^2\log d+\frac{CD}{(ab)^{e}}\\
&<\left(\frac{\sqrt{5}-1}{2}\right)^{e}(1-D)(d-1)^{2e}+\frac{CD}{(ab)^e}\\
&<\frac{1-D}{(ab)^e}+\frac{CD}{(ab)^e}\\
&<\frac{C(1-D)}{(ab)^e}+\frac{CD}{(ab)^e}\\
&=\frac{C}{(ab)^e},
\end{align*}
when $\alpha=1$;
$$\sum\limits_{m=t_1+1}^s|\zeta_m|<2^{\alpha+1}C_1(s-t_1)+2C_1d^\alpha+\frac{4C_1(d+1)^{\alpha-1}d^2}{\alpha-1}+\frac{CD}{(ab)^{e-1}(a+b)b},$$
\begin{align*}
E(a,b)&<2^{\alpha}C_1d^2+4C_1d^\alpha+\frac{8C_1(d+1)^{\alpha-1}d^2}{\alpha-1}+\frac{CD}{(ab)^{e}}\\
&<\left(\frac{\sqrt{5}-1}{2}\right)^{e}(1-D)(d-1)^{2e}+\frac{CD}{(ab)^e}\\
&=\frac{C}{(ab)^e},
\end{align*}
when $\alpha>1$.

This means that we have successfully completed the induction.

\end{proof}

\begin{rrr}
1. Let $\alpha=1$. It would be the case for $\hat{k}$.

2. The reason why this theorem does not apply for $\alpha<1$ is that in (\ref{29}),
$$\frac{1}{2^{2-\alpha}}\leq\sum_{i=1}^{c-1}\frac{1}{(i+1)^{2-\alpha}}<\frac{1}{\alpha-1}.$$
Therefore, the control over $\sum\limits_{i=1}^{c-1}(t_i-t_{i+1})|\zeta_{t_{i+1}}|$ would be $O(d^2)$. Consequently, the result would be $\sum\limits_{i=1}^{s}|\zeta_i|=O\left(\left(\frac{1}{a_0a_s}\right)^{1+\epsilon}\right)$ instead of
$\sum\limits_{i=1}^{s}|\zeta_i|=O\left(\left(\frac{1}{a_0a_s}\right)^{\frac{\alpha+1}{2}+\epsilon}\right)$ for $(a_0,a_s)\in\Delta$.

\end{rrr}

\begin{proof}[proof] (of Theorem \ref{57})
The proof is highly similar to the proof of Theorem \ref{87}. We will only highlight the differences in this proof and omit the common parts.

We need to prove that for any $e>1+\frac{\beta}{2}(1-\frac{1}{\alpha})$,
$$E(a,b)<\frac{C}{(ab)^e}$$
for $(a,b)\in\Delta$ .

Let $z=\frac{\beta}{\alpha}\in(0,1)$.

Instead of (\ref{26}), since
$$2+\beta-z=z(\alpha-1)+2=2\left(1+\frac{\beta}{2}\left(1-\frac{1}{\alpha}\right)\right)<2e,$$
we can pick $d_1>0$ such that for any $d> d_1$,
$$2^{\alpha}C_1d^2+4\gamma C_2d^{2+\beta-z}+\frac{8C_1\left(\frac{d^z}{\gamma}+2\right)^{\alpha-1}d^2}{\alpha-1}<\left(\frac{\sqrt{5}-1}{2}\right)^{e}(1-D)(d-1)^{2e}.$$

Instead of setting $d_2=2\sqrt{5}+2$, we select $d_2>0$ such that for any $d>d_2$, we have
$$\left(\frac{1}{d}\right)^z\left(\frac{1+\sqrt{5}}{2}\right)^z\leq\frac{1}{\gamma},$$
$$\frac{M}{\gamma}\left(\frac{1}{d}\right)^{1-z}\left(\frac{1+\sqrt{5}}{2}\right)^{1-z}+2\left(\frac{1}{d}\right)\left(\frac{1+\sqrt{5}}{2}\right)\leq1.$$

When $\max\{\frac{1}{a},\frac{1}{b}\}=d>d_2$, we obtain
$$a^z\leq\min\{a,b\}^z\left(\frac{1+\sqrt{5}}{2}\right)^z\leq\left(\frac{1}{d}\right)^z\left(\frac{1+\sqrt{5}}{2}\right)^z\leq\frac{1}{\gamma},$$
$$\frac{1}{\gamma}a^{1-z}+a+b\leq\frac{1}{\gamma}\min\{a,b\}^{1-z}\left(\frac{1+\sqrt{5}}{2}\right)^{1-z}+2\min\{a,b\}\left(\frac{1+\sqrt{5}}{2}\right),$$
$$\leq\frac{1}{\gamma}\left(\frac{1}{d}\right)^{1-z}\left(\frac{1+\sqrt{5}}{2}\right)^{1-z}+2\left(\frac{1}{d}\right)\left(\frac{1+\sqrt{5}}{2}\right)\leq1.$$

Therefore we have
$$2\leq\frac{1}{\gamma a^z}+1\leq\frac{1-b}{a},$$
so
\begin{equation}
2\leq\left[\frac{1}{\gamma a^z}\right]+1\leq\left[\frac{1-b}{a}\right].     \label{30}
\end{equation}

By symmetry, we also obtain
$$2\leq\left[\frac{1}{\gamma b^z}\right]+1\leq\left[\frac{1-a}{b}\right].$$

Let $d_0=\max\{d_1,d_2\}$.

When $\max\{\frac{1}{a},\frac{1}{b}\}\leq d_0$,
\begin{align*}\zeta'_i=g(a_{k-1},a_k)&<C_1\left(\left(\frac{a_{k-1}}{a_k}\right)^\alpha+\left(\frac{a_k}{a_{k-1}}\right)^\alpha\right)\\
&\leq C_1\left(\left(\frac{1}{a_k}\right)^\alpha+\left(\frac{1}{a_{k-1}}\right)^\alpha\right)
\leq2C_1d_0^\alpha\end{align*}
for $i\in[1,s]$.\\

Instead of (\ref{38}), we have
$$E(a,b)=\sum\limits_{i=1}^s|\zeta_i|\leq\sum\limits_{i=1}^s\sum\limits_{j=1}^i|\zeta'_j|<\sum\limits_{i=1}^s2C_1d_0^\alpha i,$$
$$<2C_1s^2d_0^\alpha<\frac{2C_1d_0^\alpha}{a^2b^2}\leq2C_1d_0^{4+\alpha}.$$

Let $C=\max\left\{2C_1d_0^{4+\alpha},1\right\}$.

(1) For $(a,b)\in\Delta$, when $\max\{\frac{1}{a},\frac{1}{b}\}\in[1,d_0]$, we have
$$E(a,b)< C\leq \frac{C}{(ab)^e}$$.\\

(2) Now, assume for $(a,b)\in\Delta$, when $\max\{\frac{1}{a},\frac{1}{b}\}\in[1,d-1]$, we have
$$E(a,b)< \frac{C}{(ab)^e},$$
where $d=d_0+d'$, $d'\in\mathbb{Z}^+$.

We now consider the case when $(a,b)\in\Delta$, $\max\{\frac{1}{a},\frac{1}{b}\}\in(d-1,d]$.

Instead of setting $c=\left[\frac{1-b}{a}\right]$, we let $c=\left[\frac{1}{\gamma a^z}\right]+1$. By (\ref{30}), we know that
$$c\in\left[2,\left[\frac{1-b}{a}\right]\right]$$.

Instead of (\ref{3}), we derive
\begin{align*}|\zeta_{t_i}|<C_1\left(\left(\frac{a_{t_i}}{a_0}\right)^\alpha+\left(\frac{a_0}{a_{t_i}}\right)^\alpha\right)&=
C_1\left(\left(\frac{ia+b}{a}\right)^\alpha+\left(\frac{a}{ia+b}\right)^\alpha\right)\\
&<2C_1(i+1)^\alpha.\end{align*}

Instead of (\ref{19}), we have
\begin{align*}
\sum_{i=1}^{c-1}(t_i-t_{i+1})|\zeta_{t_{i+1}}|&<\sum_{i=1}^{c-1}(t_i-t_{i+1})2C_1(i+1)^\alpha\\
&<\sum_{i=1}^{c-1}\frac{1}{(ia+b)((i+1)a+b)}2C_1(i+1)^\alpha\\
&<\sum_{i=1}^{c-1}\frac{1}{i(i+1)a^2}2C_1(i+1)^\alpha\\
&\leq\sum_{i=1}^{c-1}\frac{2}{(i+1)^2a^2}2C_1(i+1)^\alpha\\
&\leq\sum_{i=1}^{c-1}\frac{4C_1}{(i+1)^{2-\alpha}a^2} \\
&<\frac{4C_1(c+1)^{\alpha-1}}{(\alpha-1)a^2}.
\end{align*}

Instead of (\ref{7}), we have
$$\sum\limits_{m=t_c+1}^{t_1}|\zeta_m|<\frac{4C_1(c+1)^{\alpha-1}}{(\alpha-1)a^2}+\frac{CD}{(ab)^{e-1}(a+b)a}.$$

Instead of (\ref{20}), since
$$t_c<\frac{1}{a(ca+b)}<\frac{1}{ca^2}<\frac{1}{\frac{1}{\gamma a^z}a^2}=\left(\frac{1}{a}\right)^{2-z}\leq\gamma d^{2-z},$$
we have
$$|\zeta_i|<C_2\left(\left(\frac{1}{a}\right)^\beta+\left(\frac{1}{b}\right)^\beta\right)\leq2C_2d^\beta$$
for $i\in\left[1,t_c\right]$.

Instead of (\ref{10}), we can obtain
\begin{align*}
\sum\limits_{m=1}^{t_1}|\zeta_m|&=\sum\limits_{m=1}^{t_c}|\zeta_m|+\sum\limits_{m=t_c+1}^{t_1}|\zeta_m|\\
&<2t_cC_2d^\beta+\frac{4C_1(c+1)^{\alpha-1}d^2}{\alpha-1}+\frac{CD}{(ab)^{e-1}(a+b)a}\\
&<2\gamma C_2d^{2+\beta-z}+\frac{4C_1(\left[\frac{1}{\gamma a^z}\right]+2)^{\alpha-1}d^2}{\alpha-1}+\frac{CD}{(ab)^{e-1}(a+b)a}\\
&<2\gamma C_2d^{2+\beta-z}+\frac{4C_1(\frac{1}{\gamma a^z}+2)^{\alpha-1}d^2}{\alpha-1}+\frac{CD}{(ab)^{e-1}(a+b)a}\\
&\leq2\gamma C_2d^{2+\beta-z}+\frac{4C_1(\frac{d^z}{\gamma}+2)^{\alpha-1}d^2}{\alpha-1}+\frac{CD}{(ab)^{e-1}(a+b)a}.
\end{align*}

Instead of (\ref{11}), we have
$$|\zeta_s|<C_1\left(\left(\frac{a_{0}}{a_s}\right)^\alpha+\left(\frac{a_s}{a_{0}}\right)^\alpha\right)\leq2C_1\left(\frac{\sqrt{5}+1}{2}\right)^\alpha<2^{\alpha+1}C_1.$$

Instead of (\ref{15}) and (\ref{27}), we derive
\begin{align*}\sum\limits_{m=t_1+1}^s|\zeta_m|<&2^{\alpha+1}C_1(s-t_1)+2\gamma C_2d^{2+\beta-z}\\
&+\frac{4MC_1\left(\frac{d^z}{\gamma}+2\right)^{\alpha-1}d^2}{\alpha-1}+\frac{CD}{(ab)^{e-1}(a+b)b},\end{align*}
\begin{align*}
E(a,b)&<2^{\alpha}C_1d^2+4\gamma C_2d^{2+\beta-z}+\frac{8C_1\left(\frac{Md^z}{\gamma}+2\right)^{\alpha-1}d^2}{\alpha-1}+\frac{CD}{(ab)^{e}}\\
&<\left(\frac{\sqrt{5}-1}{2}\right)^{e}(1-D)(d-1)^{2e}+\frac{CD}{(ab)^e}\\
&<\frac{1-D}{(ab)^e}+\frac{CD}{(ab)^e}\\
&<\frac{C(1-D)}{(ab)^e}+\frac{CD}{(ab)^e}\\
&=\frac{C}{(ab)^e}.
\end{align*}

This means that we have already completed the induction.

\end{proof}

\begin{rrr}
1. In the second condition of Theorem \ref{57}, we only require that $$|\zeta_i|<C_2\left(\left(\frac{1}{a_0}\right)^\beta+\left(\frac{1}{a_s}\right)^\beta\right)$$
holds for $i\in\left[1,\gamma\left(\frac{1}{a_0}\right)^{2-\frac{\beta}{\alpha}}\right)$.

Since
$$s=\frac{3}{\pi^2}\frac{1}{a_0a_s}+O\left(\max\left\{\frac{1}{a_0},\frac{1}{a_s}\right\}\log\left(\min\left\{\frac{1}{a_0},\frac{1}{a_s}\right\}\right)\right),$$
therefore,
$$\frac{s}{\frac{1}{a_0a_s}}=\frac{3}{\pi^2}+O\left(\frac{\log\left(\min\left\{\frac{1}{a_0},\frac{1}{a_s}\right\}\right)}{\min\left\{\frac{1}{a_0},\frac{1}{a_s}\right\}}\right)\rightarrow\frac{3}{\pi^2}$$
as $(a_0,a_s)\rightarrow(0,0)$ in $\Delta$.

Thus, we have
$$\frac{\gamma(\frac{1}{a_0})^{2-\frac{\beta}{\alpha}}}{s}=\frac{\gamma(\frac{1}{a_0})^{1-\frac{\beta}{\alpha}}}{\frac{1}{a_s}}
\frac{\frac{1}{a_0a_s}}{s}\leq\frac{\sqrt{5}+1}{2}\frac{\gamma(\frac{1}{a_0})^{1-\frac{\beta}{\alpha}}}{\frac{1}{a_0}}
\frac{\frac{1}{a_0a_s}}{s}\rightarrow0$$
as $(a_0,a_s)\rightarrow(0,0)$ in $\Delta$.

So we know that as $(a_0,a_s)$ approaches $(0,0)$, the ratio of the length of the interval that the condition needs to be satisfied to the length of the excursion goes to $0$.

2. When $\beta\geq\alpha$, we have $1+\frac{\beta}{2}(1-\frac{1}{\alpha})\geq\frac{\alpha+1}{2}$, which indicates that Theorem \ref{56} has a better control over $\sum\limits_{i=1}^{s}|\zeta_i|$ than Theorem \ref{57}.

3. When $\beta<0$, we have
$$\frac{\gamma(\frac{1}{a_0})^{2-\frac{\beta}{\alpha}}}{s}=\frac{\gamma(\frac{1}{a_0})^{1-\frac{\beta}{\alpha}}}{\frac{1}{a_s}}
\frac{\frac{1}{a_0a_s}}{s}\geq\frac{\sqrt{5}-1}{2}\frac{\gamma(\frac{1}{a_0})^{1-\frac{\beta}{\alpha}}}{\frac{1}{a_0}}
\frac{\frac{1}{a_0a_s}}{s}\rightarrow+\infty$$
as $(a_0,a_s)\rightarrow(0,0)$ in $\Delta$.

Therefore, when $(a_0,a_s)$ is small enough such that $\gamma(\frac{1}{a_0})^{2-\frac{\beta}{\alpha}}>s$, we have
$$\sum\limits_{i=1}^{s}|\zeta_i|<sC_2\left(\left(\frac{1}{a_0}\right)^\beta+\left(\frac{1}{a_s}\right)^\beta\right)<\frac{C_2}{a_0a_s}\left(\left(\frac{1}{a_0}\right)^\beta+\left(\frac{1}{a_s}\right)^\beta\right)$$$$<2C_2\left(\frac{\sqrt{5}+1}{2}\right)^{-\frac{\beta}{2}}
\left(\frac{1}{a_0a_s}\right)^{1+\frac{\beta}{2}}.$$
Hence, we can conclude that
$$\sum\limits_{i=1}^{s}|\zeta_i|=O\left(\left(\frac{1}{a_0a_s}\right)^{1+\frac{\beta}{2}}\right).$$
Since $1+\frac{\beta}{2}<1+\frac{\beta}{2}\left(1-\frac{1}{\alpha}\right)$, this control is better than that provided by Theorem \ref{57}.

4. It is worth noting that for any $0<\beta<\alpha$ with $\alpha>1$, all the functions that satisfy the conditions of Theorem \ref{57} form a linear function space.
\end{rrr}

\section{Approximation of the RH}
\label{101}

In this section, we summarize the results from the perspective of function analysis. Then, we present some ideas about the possible path toward the RH through approximation, and we will raise some related questions for further study in \S\ref{32877}.

For any function $g$ defined on the Farey triangle, let
$$\bar{g}_n:=\frac{\sum\limits_{i=0}^{A_n-1}g\left(T^i\left(\frac{1}{n},1\right)\right)}{A_n},$$
$$\mathscr{F}_n(g):=\sum\limits_{i=1}^{A_n}\left|\sum\limits_{j=0}^{i-1}\left(g\left(T^j\left(\frac{1}{n},1\right)\right)-\bar{g}_n\right)\right|.$$

One simple fact is the triangle property.
\begin{llll}   \label{97}
$$\mathscr{F}_n(g_1+g_2)\leq\mathscr{F}_n(g_1)+\mathscr{F}_n(g_2).$$
\end{llll}

By \S\ref{96}, we know that the RH is equivalent to
$$\mathscr{F}_n(R)=O\left(n^{\frac{5}{2}+\epsilon}\right).$$

At the end of \S\ref{95}, we mention that $\left|\sum\limits_{i=1}^{A_n}\left|\theta'_{i,\lambda_2}\right|-\sum\limits_{i=1}^{A_n}\left|\theta'_{i,\lambda_1}\right|\right|=O(n^2)$.
By Theorem \ref{86}, we know that $\sum\limits_{i=1}^{A_n}|\theta'_{i,\lambda_1}|=\sum\limits_{i=1}^{A_n}|\theta'_i|=O\left(n^{2+\epsilon}\right)$;
by the definition, we know that $\sum\limits_{i=1}^{A_n}|\theta'_{i,\lambda_2}|=\mathscr{F}_n(\hat{k})$. So, we have
$$\mathscr{F}_n(\hat{k})=O\left(n^{2+\epsilon}\right).$$

Since
$$\mathscr{F}_n(R)\leq\mathscr{F}_n(\hat{k})+\mathscr{F}_n(R-\hat{k}),$$
we have the following Corollary:
\begin{ccc}
The RH is equivalent to
$$\mathscr{F}_n\left(R-\hat{k}\right)=O\left(n^{\frac{5}{2}+\epsilon}\right).$$
\end{ccc}

\begin{rrr}
What's different about $R-\hat{k}$ is that by Lemma \ref{54}, we have $-1\leq R-\hat{k}<2$. Besides, $\bar{R}_n-\bar{\hat{k}}_n=\frac{n^2}{A_n}-\frac{3A_n-1}{A_n}\approx\frac{\pi^2}{3}-3$, so when $n$ is large enough
$$\left|R-\hat{k}-\bar{R}_n+\bar{\hat{k}}_n\right|<2,$$
which means that $\sum\limits_{j=0}^{i-1}\left(R\left(T^j\left(\frac{1}{n},1\right)\right)-\hat{k}\left(T^j\left(\frac{1}{n},1\right)\right)-\bar{R}_n+\bar{\hat{k}}_n\right)$ is $2-$lipschitz. So instead of trying to control an unbounded function like $R-\frac{n^2}{A_n}\approx R-\frac{\pi^2}{3}$, we can try to control a bounded function.

Also, we need to mention that, by triangle inequality, we also know that the RH is equivalent to
$$\sum\limits_{i=1}^{A_n}\left|\sum\limits_{j=0}^{i-1}\left(R\left(T^j\left(\frac{1}{n},1\right)\right)-\hat{k}\left(T^j\left(\frac{1}{n},1\right)\right)-\frac{n^2}{A_n}+3\right)\right|=O\left(n^{\frac{5}{2}+\epsilon}\right),$$
and we also know that when $n$ is large enough,
$$\left|R-\hat{k}-\frac{n^2}{A_n}+3\right|<2,$$
which means that $\sum\limits_{j=0}^{i-1}\left(R\left(T^j\left(\frac{1}{n},1\right)\right)-\hat{k}\left(T^j\left(\frac{1}{n},1\right)\right)-\frac{n^2}{A_n}+3\right)$ is also $2-$lipschitz.
\end{rrr}

More generally, for any function $g$ defined on the Farey triangle, if $g$ satisfies the conditions of Theorem \ref{56} or Theorem \ref{57}, let $a_0=a_s=\frac{1}{n}$. Then, we have $$\sum\limits_{i=1}^{A_n}\left|\sum\limits_{j=0}^{i-1}g\left(T^j\left(\frac{1}{n},1\right)\right)\right|=O\left(n^{1+\alpha+\epsilon}\right)\text{ or } O\left(n^{2+\beta\left(1-\frac{1}{\alpha}\right)+\epsilon}\right).$$

By the condition of the theorems, we know that
$$\sum\limits_{i=0}^{A_n-1}g\left(T^i\left(\frac{1}{n},1\right)\right)<2C_1\left(\frac{\frac{1}{n}}{\frac{1}{n}}\right)^\alpha=2C_1,$$
so
$$\left|\sum\limits_{i=1}^{A_n}\left|\sum\limits_{j=0}^{i-1}\left(g\left(T^j\left(\frac{1}{n},1\right)\right)-\bar{g}_n\right)\right|-
\sum\limits_{i=1}^{A_n}\left|\sum\limits_{j=0}^{i-1}g\left(T^j\left(\frac{1}{n},1\right)\right)\right|\right|$$
$$\leq\left|\sum\limits_{i=1}^{A_n}\left|\sum\limits_{j=0}^{i-1}\bar{g}_n\right|\right|\leq
A_n\left|\sum\limits_{j=0}^{A_n-1}\bar{g}_n\right|=
A_n\left|\sum\limits_{j=0}^{A_n-1}g\left(T^j\left(\frac{1}{n},1\right)\right)\right|$$
$$<2C_1A_n=O(n^2).$$

Therefore
$$\mathscr{F}_n(g)=O\left(n^{1+\alpha+\epsilon}\right)\text{ or } O\left(n^{2+\beta\left(1-\frac{1}{\alpha}\right)+\epsilon}\right).$$

By the example suggested after Theorem \ref{56}, we know that for any $\lambda\in\mathbb{R}$, we have
$$\mathscr{F}_n\left(g_{\lambda}\right)=O\left(n^{2+\epsilon}\right).$$

By Lemma \ref{97}, we know that
\begin{ccc}
The RH is also equivalent to
$$\mathscr{F}_n\left(R-g_{\lambda}\right)=O\left(n^{\frac{5}{2}+\epsilon}\right).$$

\end{ccc}

Like Lemma \ref{54}, we have the bound for $R-g_{\lambda}$.

For $\lambda=1$,
\begin{equation}
|R-g_1|=\left|\frac{1}{ab}-\frac{1}{a}-\frac{1}{b}\right|=\frac{a+b-1}{ab}=1-\frac{(1-a)(1-b)}{ab}\leq1.   \label{99}
\end{equation}

For any $\lambda\in\mathbb{R}$, by Lemma \ref{54}, (\ref{99}), and (\ref{98}), we have
$$|R-g_{\lambda}|=|(2\lambda-1)(R-g_1)+(2-2\lambda)(R-\hat{k})|<2|2\lambda-1|+|2-2\lambda|.$$
Also, we have $\bar{R}_n-\bar{g}_{\lambda,n}\approx\frac{\pi^2}{3}-3$,  so
$$\left|R-\hat{k}-\bar{R}_n+\bar{g}_{\lambda,n}\right|<2|2\lambda-1|+|2-2\lambda|+1$$
when $n$ is large enough.

For any piecewise continuous function $f$ such that $|R-f|$ is bounded, and for any $(p,q)=1$, by Theorem \ref{106}, we have
$$\lim_{n\rightarrow\infty}\frac{pq}{A_n}\sum_{(a,b)\in e\left(\frac qn,\frac pn\right)}
(f-R)(a,b)=\int_{\Omega}(f-R)dm.$$

Using the notation of Theorem \ref{106}, since $R(a,b)=\frac{1}{ab}$, $n\geq q_{k+1}>n-q\sim n$, we know that
$$\sum_{(a,b)\in e\left(\frac qn,\frac pn\right)}R(a,b)=
\sum_{i=0}^k\frac{n^2}{q_iq_{i+1}}=n^2\left(\frac{1}{pq}+\frac{1}{pq_{k+1}}\right)\sim\frac{n^2}{pq},$$
therefore
$$\lim_{n\rightarrow\infty}\frac{pq}{A_n}\sum_{(a,b)\in e\left(\frac qn,\frac pn\right)}R(a,b)=\frac{\pi^2}{3}=\int_{\Omega}R dm,$$
so we have the following Lemma
\begin{llll}  \label{107}
For any piecewise continuous function $f$ such that $|R-f|$ is bounded, for any $(p,q)=1$, we have
$$\lim_{n\rightarrow\infty}\frac{pq}{A_n}\sum_{(a,b)\in e\left(\frac qn,\frac pn\right)}
f(a,b)=\int_{\Omega}fdm$$.

\end{llll}

One simple application is that for any $g_{\lambda}$, we have
$$\lim_{n\rightarrow\infty}\frac{pq}{A_n}\sum_{(a,b)\in e\left(\frac qn,\frac pn\right)}
g_{\lambda}(a,b)=\int_{\Omega}g_{\lambda}dm=3.$$
More generally, if functions $f_i$, $i\in[1,n]$ satisfy the conditions of Theorem \ref{56} for some $\alpha\geq1$ or Theorem \ref{57} for some $\alpha>1$ and $0<\beta<\alpha$, then for $\lambda_i\in\mathbb{R}$, $i\in[1,n]$, $\sum\limits_{i=1}^n\lambda_if_i$ also satisfies the conditions of Theorem \ref{56} for $\alpha\geq1$ or Theorem \ref{57} for $\alpha>1$ and $0<\beta<\alpha$.

If $|R-f_i|$ is bounded for $i\in[1,n]$, and $\sum\limits_{i=1}^n\lambda_i=1$, $|R-\sum\limits_{i=1}^n\lambda_if_i|$ is also bounded.

Thus all the functions that satisfy the conditions of Theorem \ref{56} for some $\alpha\geq1$ form a linear function space $\mathscr{L}_{\alpha}$. All the functions $f\in\mathscr{L}_{\alpha}$ such that $|R-f|$ is bounded form an affine function space $\mathscr{B}_{\alpha}\subset\mathscr{L}_{\alpha}$. We know that $$g_{\lambda}\in\mathscr{B}_{\alpha}.$$

Similarly, all the functions that satisfy the conditions of Theorem \ref{57} for some $\alpha>1$ and $0<\beta<\alpha$ form a linear function space $\mathscr{L}_{\alpha,\beta}$. All the functions $f\in\mathscr{L}_{\alpha,\beta}$ such that $|R-f|$ is bounded form an affine function space $\mathscr{B}_{\alpha,\beta}\subset\mathscr{L}_{\alpha,\beta}$.

\section{Further questions}
\label{32877}
1. We are wondering under what conditions, if $g_k\rightarrow g$, then we have
$$\lim_{k\rightarrow+\infty}\lim_{n\rightarrow+\infty}\frac{\mathscr{F}_n(g_k)}{n^\xi}=\lim_{n\rightarrow+\infty}\frac{\mathscr{F}_n(g)}{n^\xi}$$
for any $\xi\geq2$.

Then if we have a sequence of functions $g_k\rightarrow R$ while $\lim\limits_{n\rightarrow+\infty}\frac{\mathscr{F}_n(g_k)}{n^{\frac{5}{2}+\epsilon}}=0$, then we will get $\lim\limits_{n\rightarrow+\infty}\frac{\mathscr{F}_n(R)}{n^{\frac{5}{2}+\epsilon}}=0$ which is equivalent to the RH.

Or, if there are some other results that can relate $\mathscr{F}_n(g_k)$ with $\mathscr{F}_n(g)$?\\

2. We are also wondering if we can construct a series of functions $g_k$ which are well-controlled while $g_k\rightarrow R$. Theorem \ref{56} and \ref{57} are two useful tools to prove that some functions are well-controlled.\\

3. At the end of Chapter \ref{101}, we define the linear function space $\mathscr{L}_{\alpha}$ and $\mathscr{L}_{\alpha,\beta}$, as well as the affine function space $\mathscr{B}_{\alpha}$ and $\mathscr{B}_{\alpha,\beta}$. What are the properties of these function spaces? Can we use them to help tackle the RH?\\

4. Like Theorem \ref{56} and \ref{57}, can we obtain some other useful tools to help prove or approximate the RH given some sufficient conditions?\\

5.  After the proof of Theorem \ref{106}, we mentioned that if $\frac pq$ is irrational, we are wondering if there is a similar result to Theorem \ref{106}. We conjecture that Theorem \ref{106} still holds.

\end{document}